\documentclass[a4paper,11pt]{article}
\usepackage{amsfonts}
\usepackage{amssymb}
\usepackage{amsmath}
\usepackage{amsthm}
\usepackage{verbatim}
\usepackage{algorithmic}
\theoremstyle{plain}
\newtheorem{theorem}{Theorem}[section]
\newtheorem{lemma}[theorem]{Lemma}

\newtheorem{proposition}[theorem]{Proposition}
\theoremstyle{definition}

\usepackage{graphicx}
\usepackage{amsmath}
\usepackage{amssymb}
\usepackage{verbatim}

\begin{document}

%\begin{frontmatter}
  \title{One-dimensional fragment of first-order logic}
  \author{Lauri Hella\footnote{School of Information Sciences, University of Tampere, Finland}\, \, and\,  Antti Kuusisto\footnote{Institute of Computer Science, University of Wroc\l aw, Poland}}
%\footnote{You can put your email address or grant acknowledgement as a footnote, if you wish.}
% \address{School of Information Sciences\\ University of Tampere \\ Finland}
%\author{Antti Kuusisto}
%\address{Institute of Computer Science \\ University of Wroc\l aw \\ Poland}
 
\maketitle

\begin{abstract}
\noindent
We introduce a novel decidable fragment of first-order logic. The fragment is
\emph{one-dimensional} in the sense that
quantification is limited to applications of blocks of existential (universal) quantifiers such that at most
one variable remains free in the quantified formula.
The fragment is closed under Boolean operations,
but additional restrictions (called \emph{uniformity conditions})
apply to combinations of atomic formulae with two or more variables.
We argue that the notions of \emph{one-dimensionality} and
\emph{uniformity} together offer a novel perspective
on the \emph{robust decidability} of modal logics.
%The technique is a central contribution of the article.
%We also show that the one-dimensional fragment is expressively equivalent to a
%polyadic modal logic with the capacity of
%permuting and forming Boolean combinations of
%accessibility relations.
%Our decidability proof is based on a 
%novel proof technique involving a satisfiability
%preserving translation into monadic first-order logic.
We also establish that minor modifications to the restrictions
of the syntax of the one-dimensional fragment lead to undecidable 
formalisms. Namely, the \emph{two-dimensional} 
and \emph{non-uniform one-dimensional} fragments
are shown undecidable. Finally, we prove that
with regard to expressivity,
the one-dimensional fragment is
incomparable with both the
guarded negation fragment and 
two-variable logic with counting. 
Our proof of the decidability of the one-dimensional fragment is
based on a technique involving a direct reduction to the monadic
class of first-order logic.
The novel technique is itself of an independent mathematical interest.
\end{abstract}

%\begin{keyword}
%Extensions of modal logic,
%fragments of first-order logic,
%Boolean modal logic, decidability
%\end{keyword}
% \end{frontmatter}

%%%%%%%%%%%%%%my own stuff
\newcommand{\M}{\mathfrak{M}}
\newcommand{\T}{\mathfrak{T}}
\newcommand{\A}{\mathfrak{A}}
\newcommand{\B}{\mathfrak{B}}
\newcommand{\N}{\mathfrak{N}}
%%%%%%%%%%%%%%%%%%

%
%%%%%%%%%%%%%%%%%%%%%%%%%%%%%%%%%%%%%%%%%%%%
%

%
\section{Introduction}
Decidability questions constitute one of the core themes in
computer science logic. 
Decidability properties of several fragments of first-order logic have
been investigated after the completion of the program
concerning the classical decision problem.
Currently perhaps the most important two frameworks studied in this context
are those based on the \emph{guarded fragment}
\cite{IEEEonedimensional:andreka}
and \emph{two-variable logics}.
Two-variable logic $\mathrm{FO}^2$ was introduced by Henkin
in \cite{IEEEonedimensional:henkin}
and showed decidable in \cite{IEEEonedimensional:mortimer} by Mortimer.
%
%The equality-free two-variable
%
%logic was shown decidable already in and \cite{IEEEonedimensional:mortimer}.
%
The satisfiability and finite satisfiability problems of 
two-variable logic were proved to
be $\mathrm{NEXPTIME}$-complete in \cite{IEEEonedimensional:gradelkolaitis}.
The extension of two-variable logic with counting quantifiers, $\mathrm{FOC}^2$,
was shown decidable in \cite{IEEEonedimensional:gradel}.
It was subsequently proved to be 
$\mathrm{NEXPTIME}$-complete in \cite{IEEEonedimensional:pratth}.
%
%\cite{IEEEonedimensional:pacholski} that
%
%$\mathrm{FOC}^2$ is in $\mathrm{2NEXPTIME}$
%
%that $\mathrm{FOC}^2$ is
%

%
Research concerning
decidability of variants of 
two-variable logic has been
very active in recent years.
Recent articles in the field include for example
\cite{IEEEonedimensional:kieronski}
\cite{IEEEonedimensional:charatonic},
%
%\cite{IEEEonedimensional:kieronskimichaliszyn},
%
\cite{IEEEonedimensional:kieronskitendera},
%
%\cite{IEEEonedimensional:zeume},
%
\cite{IEEEonedimensional:tendera},
and several others.
The recent research efforts have mainly concerned decidability
and complexity issues
in restriction to particular classes of 
structures, and also questions
related to different built-in features and operators that 
increase the expressivity of the base language.
Guarded fragment $\mathrm{GFO}$ was originally conceived in
\cite{IEEEonedimensional:andreka}.
It is a restriction of first-order logic
that only allows quantification of
``guarded" new variables---a restriction
that makes the logic rather similar
to modal logic.
The guarded fragment has generated a vast
literature, and several related decidability
questions have been studied.
The fragment has recently been significantly
generalized in \cite{IEEEonedimensional:barany}.
The article introduces the
\emph{guarded negation first-order logic} $\mathrm{GNFO}$.
This logic only allows
negations of formulae that are guarded
in the sense of the guarded fragment.
The guarded negation fragment
has been shown complete for $\mathrm{2NEXPTIME}$
in \cite{IEEEonedimensional:barany}.
Two-variable logic and guarded-fragment are
examples of decidable fragments of first-order logic
that are not based on restricting the quantification
patterns of formulae, unlike the prefix classes studied in the context of
the classical decision problem.
Surprisingly, not many such frameworks have been investigated in
the literature.
%
%In fact related research has concentrated on extensions of
%
%two-variable logics and the guarded fragment.
%

%
In this paper we introduce a novel
decidable fragment that allows arbitrary quantifier
alternation patterns.
The \emph{uniform one-dimensional
fragment} $\mathrm{UF}_1$ of first-order logic is obtained by restricting
quantification to blocks of existential (universal) quantifiers
that \emph{leave
at most one free variable} in the resulting formula.
Additionally, a \emph{uniformity condition} applies to the use of
atomic formulae: if $n,k\geq 2$, then a Boolean combination of
atoms $R(x_1,...,x_k)$ and $S(y_1,...,y_n)$
is allowed only if $\{x_1,...,x_k\} = \{y_1,...,y_n\}$.
Boolean combinations of formulae with at most one
free variable can be formed freely.
We establish decidability of
the satisfiability and finite satisfiability problems of $\mathrm{UF}_1$.
We also show that if the uniformity condition is
lifted, we obtain an undecidable logic. Furthermore, if
we keep uniformity but go two-dimensional by allowing existential (universal)
quantifier blocks that leave two variables free, we again obtain an
undecidable formalism. Therefore, \emph{if we lift either of the two
restrictions that our fragment is based on, we 
obtain an undecidable logic.}
In addition to studying decidability,
we also show that $\mathrm{UF}_1$ is
incomparable in expressive power with both
$\mathrm{FOC}^2$ and
$\mathrm{GNFO}$.
%
%We also establish the analogous results concerning
%
%fluted logic \cite{IEEEonedimensional:purdy}.
%
%Fluted logic is an example of the very few well-studied fragments
%
%of first-order logic that are
%
%not based on quantifier prefix restrictions,
%
%and are known to be decidable.
%

%
In \cite{IEEEonedimensional:vardi}, Vardi initiated an intriguing research effort
that aims to understand phenomena behind the
\emph{robust decidability} of different variants of
modal logic.
In addition to \cite{IEEEonedimensional:vardi},
see also for example \cite{IEEEonedimensional:gradelrobustly} and
the introduction of \cite{IEEEonedimensional:barany}.
Modal logic indeed has several features related to
what is known about decidability. In particular,
modal logic embeds into both $\mathrm{FO}^2$
and $\mathrm{GFO}$.
However, there exist several important and widely applied decidable extensions
of modal logic that
do not embed into \emph{both} $\mathrm{FO}^2$ and
$\mathrm{GFO}$.
%
%but are indeed contained in $\mathrm{UF}_1$.
%
Such extensions include \emph{Boolean modal logic}
(see \cite{IEEEonedimensional:passy}, \cite{IEEEonedimensional:lutz})
and basic \emph{polyadic modal logic}, i.e, modal logic
containing accessibility relations of 
arities higher than two (see \cite{IEEEonedimensional:venema}).
Boolean modal logic allows Boolean combinations of
accessibility relations and therefore can express for
example the formula $\exists y\bigl(\neg R(x,y) \wedge P(y)\bigr)$.
Polyadic modal logic can express the
%
%first-order
%
formula $\exists x_2...\exists x_k\bigl(R(x_1,...,x_k) \wedge P(x_2)\wedge...\wedge P(x_k)\bigr)$.
Boolean modal logic and polyadic modal logic
are both inherently one-dimensional, and furthermore, satisfy the uniformity
condition of $\mathrm{UF}_1$. Both logics embed
into $\mathrm{UF}_1$.
The notions of \emph{one-dimensionality} and \emph{uniformity}
can be seen as novel features that can help, in part, \emph{explain} decidability
phenomena concerning modal logics.
Importantly, also the equality-free fragment of $\mathrm{FO}^2$
embeds into $\mathrm{UF}_1$.
In fact, when attention is restricted to vocabularies with
relations of arities at most two,
then the expressivities of
$\mathrm{UF}_1$ and the equality-free fragment of $\mathrm{FO}^2$ coincide.
Instead of seeing this as a weakness of $\mathrm{UF}_1$, we in fact regard $\mathrm{UF}_1$ as
\emph{a canonical generalisation} of (equality-free) $\mathrm{FO}^2$ into
contexts with arbitrary relational vocabularies.
The fragment $\mathrm{UF}_1$ can be regarded as a
\emph{vectorisation} of $\mathrm{FO}^2$
that offers new possibilities for extending research efforts
concerning two-variable logics.
\emph{It is worth noting that for example in database theory
contexts, two-variable logics as such are not
always directly applicable due to the arity-related limitations.}
Thus we believe that the one-dimensional fragment is indeed a
worthy discovery that
\emph{extends the scope of research on
two-variable logics to the realm involving relations of arbitrary arities.}
%

%
%Thus $\mathrm{UF}_1$ introduce a novel
%
%natural way of extending two-variable logics
%
%with the possibility of obtaining decidable fragments. 
%
%In this context it is worth noting that $\mathrm{UF}_1$ copes fine
%
%with relation symbols of any arity. This is of course not the case
%
%with two-variable logic, and therefore two-variable logics as such are not
%
%directly applicable in typical database theory contexts, for example.
%

%
Instead of extending basic techniques from the field of
two-variable logic, our decidability proof is based on a direct
satisfiability preserving translation of 
$\mathrm{UF}_1$ into monadic first-order logic.
The novel proof technique is mathematically interesting in its own right, and is
in fact a central contribution of this article;
the proof technique is clearly robust and can be modified and extended to give other
decidability and complexity results. Furthermore, as a by-product of our
proof, we identify a natural polyadic modal logic $\mathrm{MUF}_1$, which is expressively
equivalent to the one-dimensional fragment. This \emph{modal normal form} for
the one-dimensional fragments is also---we believe---a nice contribution.
%
%The structure of the decidability proof is outlined in Section \ref{outline}.
%
%The proof is based on an encoding of information concerning $k$-ary
%
%relations by unary relations and a certain hypertorus construction.
%
%The proof method is novel and robust, 
%
%and can
%
%quite likely
%
%be easily modified to give other decidability
%
%results concerning extensions of two-variable logic with
%
%one-dimensional quantification. 
%
%We aim to give rigorous and fully self-contained proofs of all our results.
%

%
%The classical decision problem has been solved, but  a lot can and 
%
%should be done with first-order fragments not based on quantifier prefixes.
%
%We aim to take some steps into this direction.
%

%
\section{Preliminaries}
Let $\mathbb{Z}_+$ denote the set of positive integers.
Let $\mathcal{T}$ denote a \emph{complete
relational vocabulary}, i.e., $\mathcal{T} := \bigcup_{k\, \in\, \mathbb{Z}_+} \tau_k$,
where $\tau_k$ denotes a countably infinite
set of $k$-ary relation symbols.
Each vocabulary $\tau$ we consider below
is assumed to be a subset of $\mathcal{T}$.
%
%Let $\tau\subseteq\mathcal{T}$.
%
A $\tau$-formula of first-order logic is a formula whose set of non-logical
symbols is a subset of $\tau$. A $\tau$-model
is a model whose set of interpreted non-logical
symbols is $\tau$.
Let $\mathrm{VAR}$ denote the countably
infinite set $\{\, x_i\ |\ i\in\mathbb{Z}_+\ \}$
of \emph{variable symbols}.
We define the set of $\mathcal{T}$-formulae
of first-order logic in the usual way,
assuming that all variable symbols 
are from $\mathrm{VAR}$.
Below we use \emph{meta-variables}
$x,y,z$ in order
to denote variables in $\mathrm{VAR}$.
Also symbols of the type $y_i$ and $z_i$, where $i\in\mathbb{Z}_+$, will be used 
as meta-variables.
In addition to meta-variables, we also need to directly use the variables $x_i\in\mathrm{VAR}$ below.
Note that for example the meta-variables $y_1$ and $y_2$ may denote the
same variable in $\mathrm{VAR}$, while the 
variables $x_1,x_2\in\mathrm{VAR}$ of course
simply \emph{are} different variables.
Let $R$ be a $k$-ary relation symbol,  $k\in\mathbb{Z}_+$.
An atomic formula $R(y_{1},...,y_{k})$ is called \emph{$m$-ary}
if there are exactly $m$ distinct variables in the set $\{y_{1},...,y_{k}\}$.
For example, if $x,y$ are distinct variables, then $S(x,y)$ and $T(y,x,y,y)$
are binary, and $U(x_1,x_6,x_3,x_2,x_1,x_6)$ is $4$-ary.
An \emph{$m$-ary $\tau$-atom} is an atomic formula
that is $m$-ary, and the relation symbol of the formula is in $\tau$.
Let $\tau \subseteq \mathcal{T}$.
Let $\M$ a $\tau$-model
with the domain $M$.
A function $f$ that maps some subset of $\mathrm{VAR}$ into $M$
is an \emph{assignment}.
Let $\varphi$ be a $\tau$-formula with the free
variables $y_1,...,y_k$. Let $f$ be an assignment
that interprets the free variables of $\varphi$ in $M$.
We write $\M,f\models\varphi$ if $\M$ satisfies $\varphi$
when the free variables of $\varphi$ are interpreted according to $f$.
%
%Two $\tau$-formulae $\varphi$ and $\psi$ of first-order logic are \emph{equivalent},
%
%if the equivalence $\M,f\models\varphi\ \Leftrightarrow\ \M,f\models\psi$ holds
%
%for all $\tau$-models $\M$ and assignments $f$ that interpret the free variables in both formulae.
%
Let $u_1,...,u_k\in M$. Let  $\varphi$ be a $\tau$-formula whose
free variables are among $y_1,...,y_k$.
We write $\M,\frac{(u_1,...,u_k)}{(y_1,...,y_k)}\models \varphi$
if $\M,f\models\varphi$ for some assignment $f$ such
that $f(y_i) = u_i$ for each $i\in\{1,..,k\}$.
By a \emph{non-empty conjunction} we 
mean a finite conjunction with at least one 
conjunct; for example $R(x,y) \wedge \exists y P(y)$ and $\top$
are non-empty conjunctions.
%

%
%Let $\chi$ and $\chi'$ be formulae
%
%such that the sets of 
%
%non-logical symbols occurring
%
%in $\chi$ and $\chi'$ 
%
%are $\sigma\subseteq \mathcal{T}$
%
%and $\sigma'\subseteq\mathcal{T}$, respectively.
%
%Let $U$ and $U'$ be the sets of free variable symbols
%
%of  $\chi$ and $\chi'$, respectively.
%
%As one would expect, the formulae $\chi$ and $\chi'$ are 
%
%\emph{equivalent}, if the equivalence $\M,f\models\chi\ \Leftrightarrow\ \M,f\models\chi'$
%
%holds for all $(\sigma\cup\sigma')$-models $\M$ and 
%
%assignments $f$ that interpret the set $U\cup U'$ of variables
%
%in the domain of $\M$.
%

%
By \emph{monadic first-order logic}, or $\mathrm{MFO}$, we mean the
fragment of first-order logic \emph{without equality},
where formulae contain only unary relation symbols. 
Let $k\in\mathbb{Z}_+$. A \emph{$k$-permutation} is a
bijection $\sigma:\{1,...,k\}\rightarrow\{1,...,k\}$.
When $k$ is irrelevant or clear
from the context, we simply talk about
permutations.
%
%instead of $k$-permutations.
%
%If $(u_1,...,u_k)$ is a tuple, then
%
%any tuple $(u_{\mu(1)},...,u_{\mu(k)})$,
%
%where $\mu$ is some $k$-permutation,
%
%is called a \emph{permutation of
%
%$(u_1,...,u_k)$}.
%

%
Let $k\in\mathbb{Z}_+$.
We let $(u,...,u)_k$ and $u^k$ denote the
$k$-tuple containing $k$ copies of the object $u$.
When $k=1$, this tuple is identified with the object $u$.
Let $l$ and $k\leq l$ be positive integers.
Let $S$ be a set, and let $(s_1,...,s_l)\in S^l$ be a tuple.
We let $(s_1,...,s_l)\upharpoonright k$ denote
the tuple $(s_1,...,s_k)$.
Let $R\subseteq S^l$ be an $l$-ary relation.
We let $R\upharpoonright k$
denote the $k$-ary relation $R'\subseteq S^k$
defined such that for each $(s_1,...,s_k)\in S^k$,
we have $(s_1,...,s_k)\in R'$ iff
$(s_1,...,s_k) = (u_1,...,u_l)\upharpoonright k$ for 
some tuple $(u_1,...,u_l)\in R$.
Recall that $\bigwedge\emptyset$ is assumed to be always true,
while $\bigvee\emptyset$ is always false.
\section{The one-dimensional fragment}\label{sectionthree}
We shall next define the \emph{uniform one-dimensional}
fragment $\mathrm{UF}_1$ of first-order logic.
Let $Y = \{y_1,...,y_n\}$ be a set of variable symbols,
%
% such that $|Y|\geq 2$.
%
and let $R$ be a
$k$-ary
relation symbol.
%
%where $k\geq 2$.
%
An atomic formula $R(y_{i_1},...,y_{i_k})$ is called a \emph{$Y$-atom}
if $\{y_{i_1},...,y_{i_k}\} = Y$.
A finite set of $Y$-atoms is called a \emph{$Y$-uniform set}.
When $Y$ is irrelevant or known from the context, we may
simply talk about a \emph{uniform set}.
%
%If $\tau$ is a vocabulary
%
%and $R(y_{i_1},...,y_{i_k})$ is $\tau$-atom, where $R$ belongs
%
%
%
For example, assuming that $x,y,z$ are distinct variables,
$\{T(x,y), S(y,x)\}$ and $\{R(x,x,y),R(y,y,x), S(y,x)\}$ are uniform sets,
while $\{R(x,y,z),R(x,y,y)\}$
%
%and $\{R(y,y,y), S(y,y)\}$
%
is not.
The empty set is a $\emptyset$-uniform set.
Let $\tau\subseteq\mathcal{T}$.
The set $\mathrm{UF}_1(\tau)$, or the set of $\tau$-formulae of
the one-dimensional fragment, is
the smallest set $\mathcal{F}$ satisfying the following conditions.
\begin{enumerate}
%
%\item
%$\bot,\top\in \mathcal{F}$.
%
\item
Every unary $\tau$-atom is in $\mathcal{F}$, and $\bot,\top\in \mathcal{F}$.
\item
If $\varphi\in \mathcal{F}$, then $\neg\varphi\in\mathcal{F}$.
If $\varphi_1,\varphi_2\in \mathcal{F}$,
then $(\varphi_1\wedge\varphi_2)\in \mathcal{F}$.
\item
Let $Y = \{y_1,...,y_k\}$ be a set of variable symbols.
Let $U$ be a finite set of
formulae $\psi\in\mathcal{F}$ whose
free variables are in $Y$.
Let $V\subseteq Y$.
Let $F$ be a $V$-uniform set of $\tau$-atoms.
Let $\varphi$ be any Boolean combination of formulae in $U\cup F$.
Then $\exists y_2...\exists y_k\, \varphi\, \in \mathcal{F}$.
\item
If $\varphi\in\mathcal{F}$,
%
%has the free variable $y$,
%
then $\exists y\, \varphi\, \in\, \mathcal{F}$.
\end{enumerate}
Notice that there is no equality symbol in the language.
Notice also that the formation rule (iv)  is strictly speaking not needed
since the rule (iii) covers it. Concerning the rule (i),
notice that also atoms of the type $S(x,...,x)_k$, where $k\not=1$,
are legitimate formulae.
Let $\mathrm{UF}_1$ denote the set $\mathrm{UF}_1(\mathcal{T})$.
%
%of all sets $\mathrm{UF}_1(\tau)$, where $\tau$ is a
%
%finite subset of $\mathcal{T}$.
%

%
\subsection{Intuitions underlying the decidability proof}\label{outline}
We show decidability of the satisfiability and 
finite satisfiability problems of $\mathrm{UF}_1$
by translating $\mathrm{UF}_1$-formulae
into equisatisfiable $\mathrm{MFO}$-formulae.
We first translate $\mathrm{UF}_1$ into a
logic $\mathrm{DUF}_1$. This logic
is a normal form for $\mathrm{UF}_1$ such that all literals of arities
higher than one appear in simple conjunctions, as for example
in the formula $\exists y \exists z\bigl( R(x,z,y,z)\wedge\neg S(y,x,z) \wedge \varphi(y)\bigr)$.
The logic $\mathrm{DUF}_1$ is then translated into a
modal logic $\mathrm{MUF}_1$, which is an
essentially variable-free formalism for $\mathrm{DUF}_1$.
%
%The fact that $\mathrm{UF}_1$ has such a modal representation is a
%
%by-product of our proof and one of the contributions of this paper.
%
%We conjecture that the complexity of the satisfiability problem of $\mathrm{MUF}_1$ is
%
%different from that of $\mathrm{UF}_1$, but for the sake of simplicity,
%
%we shall not take up complexity issues in this article.
%
In Section \ref{decidabilitysection}
we show how formulae of the logic $\mathrm{MUF}_1$ are
translated into equisatisfiable formulae of $\mathrm{MFO}$,
which is well-known to have the finite model property.
%
%We shall describe the intuition behind the use of the
%
%hypertorus below when presenting the related arguments,
%
%but the overall idea is that 
%

%
The semantics of $\mathrm{MUF}_1$  is
defined (see Section \ref{modal}) 
with respect to pointed models $(\M,u)$,
where $u\in M = \mathit{Dom}(\M)$.
%
%Each $\mathrm{MUF}_1$-formula $\varphi$ canonically defines
%
If $\varphi$ is a formula of $\mathrm{MUF}_1$, we let
$\Vert\varphi\Vert^{\M}$ denote
the set $\{\, v\in M\ |\ (\M,v)\models\varphi\, \}$.
In Section \ref{decidabilitysection}
we fix a $\mathrm{MUF}_1$- formula $\psi$
and translate it to an $\mathrm{MFO}$-formula $\psi^*(x)$.
We prove that if $(\M,v)\models\psi$, then $\psi^*(x)$ is satisfied in a
model $\T$, whose domain is $M\times T$,
where $T$ is the domain of an \emph{$m$-dimensional hypertorus
of arity $l$}. Such a hypertorus is a structure
$(T,R_1,...,R_m)$, where the $m$ different
relations $R_i$ are all $l$-ary.
Intuitively, the domain of $\T$ consists of several copies of $M$,
one copy for each point of the hypertorus.
Let $\mathrm{SUB}_{\psi}$ denote the set of subformulae of $\psi$.
The vocabulary of $\T$ consists of monadic predicates $P_{\alpha}$ and $P_t$,
where $\alpha\in\mathrm{SUB}_{\psi}$ and $t\in T$.
The predicates are interpreted such that $P_{\alpha}^{\T} := \Vert\alpha \Vert^{\M}\times T$
and $P_t^{\T} := M\times\{t\}$.
%
%Intuitively, the predicate $P_{\alpha}$ encodes information
%
%about the set $\Vert\alpha\Vert^{\M}$.
%
%The predicates $P_t$
%
%These predicates are used in order to
%
%encode information about the
%
%\emph{accessibility relations} that occur
%
%in the modal formula $\psi$.
%
%The accessibility relations are, essentially,
%
%conjunctions of literals of some arity $k\geq 2$.
%

%
\begin{comment}
%
We will give a rigorous and fully self-contained 
%
proof of the decidability of $\mathrm{UF}_1$,
%
but to get a proper grasp of the background
%
intuition, the reader is assumed to be familiar with
%
the notion of a \emph{largest filtration}
%
(see \cite{IEEEonedimensional:venema}).
%
\end{comment}
%

%
We will give a \emph{rigorous and self-contained}
proof of the decidability of $\mathrm{UF}_1$,
but to get an (admittedly very rough) initial idea of
some of the related background
intuitions, consider the following construction.
(\emph{It may also help to refer back to this section while
internalizing the proof.})
Consider a formula of ordinary unimodal logic $\varphi$
and a Kripke model $\mathfrak{N}$.
We can \emph{maximize} the accessibility relation $R$ of $\mathfrak{N}$
by defining a new relation $S\subseteq N\times N$ such
that $(u,v)\in S$ iff for all formulae $\Diamond\beta\in\mathrm{SUB}_{\varphi}$,\vspace{1mm}
we have
\vspace{1mm}
$\hspace{3.5cm}(\N,v)\models\beta\, \Rightarrow\, (\N,u)\models\Diamond\beta.\hfill (1)$\\
If we replace $R$ by $S$ in $\mathfrak{N}$, then each point $w$
in the new model will satisfy exactly the same subformulae of $\varphi$
as $w$ satisfied in the old model. 
Thus we can \emph{encode information concerning $R$}
by using the (so-called filtration) condition given by Equation 1.
The equation talks about the \emph{sets} $\Vert\beta\Vert^{\mathfrak{N}}$
and $\Vert\Diamond\beta\Vert^{\mathfrak{N}}$, and thus
it turns out that we can encode the information given by the
equation by \emph{monadic predicates} $P_{\beta}$ and $P_{\Diamond\beta}$
corresponding to the sets $\Vert\beta\Vert^{\mathfrak{N}}$
and $\Vert\Diamond\beta\Vert^{\mathfrak{N}}$ (cf. the 
formulae $\mathit{PreCons}_{\delta}$ and $\mathit{Cons}_{\delta}$
in Section \ref{translationsection}). \emph{This way we can
encode information concerning accessibility relations
by using formulae of $\mathrm{MFO}$.}
%
\begin{comment}
%
The translation involves a model construction
%
based on an \emph{$m$-dimensional hypertorus
%
of arity $l$}. This is a structure
%
$(T,R_1,...,R_m)$, where the $m$ different
%
relations $R_i$ are all $l$-ary.
%
If $\psi$ is satisfied by a pointed model $(\M,u)$, the formula
%
$\psi^*(x)$ will be satisfied in a model $\T$ that
%
essentially contains a copy of $\M$ for each point of the hypertorus.
%
%
%
The modal \emph{accessibility relations} occurring in $\psi$
%
will essentially be conjunctions of $k$-ary literals, $k\geq 2$.
%
Information concerning the accessibility relations of $\psi$
%
will be encoded in the model $\T$ with the help of unary
%
predicates that encode the structure of the relations $R_i$.
%
The translation involves a model construction
%
based on an \emph{$m$-dimensional hypertorus
%
of arity $l$}. This is a structure
%
$(T,R_1,...,R_m)$, where the $m$ different
%
relations $R_i$ are all $l$-ary.
%
\end{comment}
%

%
This construction does not work if one tries to maximize \emph{both} a
binary relation $R$ \emph{and its complement} $\overline{R}$
at the same time: the problem is that the maximized
relations $S$ and $\overline{S}$ will not necessarily be 
complements of each other. For this reason we need to \emph{make enough room} for 
maximizing accessibility relations. Below we will simultaneously maximize several types of
accessibility relations that cannot be allowed to intersect. 
Thus we need to use an \emph{$n$-dimensional}
hypertorus (rather than a usual 2D torus).
Each $k$-ary accessibility relation type $\delta$ of the
translated $\mathrm{MUF}_1$-formula
will be reserved a sequence $\overline{r} := (M\times\{t_1\},..., M\times\{t_k\})$
of copies of $M$ from the domain of $\T$.
Information concerning $\delta$ will be
encoded into this sequence $\overline{r}$ of models.

\subsection{Diagrams}
%

%
%%%%%%%%%%%%%%%%%%%%%%%%%%%%%%%%%%
%
%\subsection{Diagrams}
%

%
Let $\tau\subseteq\mathcal{T}$ be a \emph{finite} vocabulary.
%
%Assume that $\tau$ contains at least one relation symbol of some
%
%arity higher than one.
%
Let $k\geq 2$ be an integer,
and let $Y = \{y_1,...,y_k\}$ be a set of \emph{distinct} variable
symbols. A \emph{uniform $k$-ary $\tau$-diagram}
is a maximal satisfiable set of $Y$-atoms and negated
$Y$-atoms of the vocabulary $\tau$.
(The empty set is \emph{not} considered to be a uniform $k$-ary $\tau$-diagram;
this case is relevant when $\tau$ contains no relation
symbols of the arity $k$ or higher.)
For example, let $\tau = \{P, R, S\}$, where the arities
of $P$, $R$, $S$ are $1$, $2$, $3$, respectively.
Now
$
\{ R(x,y),\neg R(y,x),$ $
S(y,x,x),$ $S(x,y,x), \neg S(x,x,y),
S(x,y,y),$ $ \neg S(y,x,y), S(y,y,x)\}
$
is a uniform binary $\tau$-diagram.
Here we assume that $x$ and $y$ are distinct
variables.
%
\begin{comment}
%
The set
%
%
%
$
%
\{S(x,y,z), S(x,z,y),$ $ \neg S(y,x,z),
%
S(y,z,x), \neg S(z,x,y), S(z,y,x)\}
%
$
%
%
%
is a uniform ternary $\tau$-diagram.
%
Again $x,y,z$ are assumed distinct.
%
\end{comment}
%

%
Let $\tau\subseteq\mathcal{T}$ be a \emph{finite} vocabulary.
The set $\mathrm{DUF}_1(\tau)$  is
the smallest set $\mathcal{F}$ satisfying the following conditions.
\begin{enumerate}
\item
Every unary $\tau$-atom is in $\mathcal{F}$.
Also $\bot,\top\in \mathcal{F}$. 
\item
If $\varphi\in \mathcal{F}$, then $\neg\varphi\in\mathcal{F}$.
If $\varphi_1,\varphi_2\in \mathcal{F}$,
then $(\varphi_1\wedge\varphi_2)\in \mathcal{F}$.
\item
Let $\delta$ be a uniform $k$-ary $\tau$-diagram in 
the variables $y_1$,...,$y_k$, where $k\geq 2$. Let $\varphi$ be a
non-empty conjunction of a finite set $U$ of formulae in $\mathcal{F}$
%
%such that each formula in $U$ has at most one free variable, and
%
whose free variables are among $y_1$,...,$y_k$.
Then $\exists y_2...\exists y_k\, \bigl(\, \bigwedge\delta\, \wedge\, \varphi\, \bigr)\, \in\, \mathcal{F}$.
%
%Also $\exists y_2...\exists y_k\, \delta\, \in \mathcal{F}$.
%
\item
If $\varphi\in\mathcal{F}$
has at most one free variable, $y$, 
then $\exists y\, \varphi\in \mathcal{F}$.
\end{enumerate}
Let $\mathrm{DUF}_1$ denote the set of exactly
all formulae $\varphi$ such that for some finite $\tau\subseteq\mathcal{T}$,
we have $\varphi\in\mathrm{DUF}_1(\tau)$.
$\mathrm{UF}_1$ translates effectively into $\mathrm{DUF}_1$;
see the appendix for the proof.
Here we briefly \emph{sketch} the principal idea behind the translation.
Consider a $\mathrm{UF}_1$-formula $\exists\overline{y}\, \psi$, where
$\overline{y}$ denotes a tuple of variables.
Put $\psi$ into disjunctive normal form $\psi_1\vee...\vee\psi_k$.
Thus $\exists\overline{y}\, \psi$ translates into the formula
$\exists\overline{y}\, \psi_1\vee...\vee\exists\overline{y}\, \psi_k$,
where the formulae $\psi_i$ are conjunctions.
Each conjunction $\psi_i$ is equivalent to a disjunction $\psi_{i,1}\vee...\vee\psi_{i,m}$,
where $\psi_{i,j}$ is of the desired type $\bigl(\, \bigwedge\delta\wedge\varphi\, \bigr)$.
\subsection{Hypertori}
We next define a class of hypertori. It may help to have a look
at Lemma \ref{toruslemma} before internalizing the definition.
Let $l\geq 2$ and $n\geq 2$ be integers.
%
%Let $m\in\{1,...,n\}$, and
%
Define 
$T := \{1,...,n\} \times \{1,...,l\} \times\{0,1,2\}.$
Let $(t_1,...,t_l)\in T^l$ be a tuple.
Let $t_1 = (m,m',m'')$.
Let $j\in \{1,...,n\}$.
For each $i\in \{2,...,l\}$,
let $t_i = (p,p',p'')$ such that
the following conditions hold.
\begin{enumerate}
\item
%Let $t_1 = (m,m',m'')$.
%
%Let $i\in \{2,...,l\}$ and $t_i = (p,p',p'')$.
%
$p - m \equiv j-1 \mod n$.
\item
%Let $t_1 = (m,m',m'')$.
%
%Let $i\in \{2,...,l\}$ and $t_i = (p,p',p'')$.
%
$p' - m' \equiv i-1 \mod l$.
\item
%Let $t_1 = (m,m',m'')$.
%
%Let $i\in \{2,...,l\}$ and $t_i = (p,p',p'')$.
%
$p'' - m''\equiv 1 \mod 3$.
\end{enumerate}
Let us call such a tuple $(t_1,...,t_l) \in T^l$ the
\emph{$j$-th good $l$-ary sequence originating
from $t_1$}.
Define the relation $R_j \subseteq T^l$ such that
$(s_1,...,s_l)\in R_j$ iff $(s_1,...,s_l)$ is the $j$-th
good $l$-ary sequence originating from $s_1$.
We call the structure $\bigl(T,R_1,...,R_n\bigr)$ the
\emph{$n$-dimensional
hypertorus of the arity $l$}.
%
%The following lemma is straightforward to prove.

%
\begin{lemma}\label{toruslemma}
Let $\bigl(T,R_1,...,R_n\bigr)$ be an $n$-dimensional
hypertorus of the arity $l$. 
Let $j\in \{1,...,n\}$ and $k\in\{2,...,l\}$.
Then the following conditions hold.
\begin{enumerate}
\item
For each $t\in T$, 
there exists exactly one tuple $(s_1,...,s_k)\in R_j\upharpoonright k$
such that $t = s_1$. We have $s_i\not = s_j$
for all $i,j\in\{1,...,k\}$ such
that $i\not= j$.
\item
Let $(s_1,...,s_k)\in R_j\upharpoonright k$.
Let $\sigma$ be a $k$-permutation, and let $i\in\{1,...,n\}\setminus\{j\}$.
Then $(s_{\sigma(1)},...,s_{\sigma(k)})\not\in R_i\upharpoonright k$.
\item
Let $(s_1,...,s_k)\in R_j\upharpoonright k$.
Let $\mu$ be any $k$-permutation other than the identity permutation.
Then $(s_{\mu(1)},...,s_{\mu(k)})\not\in R_j\upharpoonright k$.
\end{enumerate}
\end{lemma}
\begin{proof}
Straightforward.
\end{proof}
In the rest of the article, we let $\T(n,l)$ denote the $n$-dimensional
hypertorus of the arity $l$. We let $T(n,l)$ and $R_j(n,l)$ denote,
respectively, the domain and the relation $R_j$ of $\T(n,l)$.
%

%
%%%%%%%%%%%%%%%%%%%%%%%%%%%%%%%%%%%%%%
%

%
\subsection{Translation into a modal logic}\label{modal}
Let $\tau\subseteq\mathcal{T}$ be a \emph{finite} vocabulary,
and let $k\geq 2$ be an integer.
Let $\M$ be a $\tau$-model with the domain $M$.
Let $\delta$ be a uniform $k$-ary $\tau$-diagram
in the variables $x_1,...,x_k$.
Notice that here we 
use the standard variables $x_1,...,x_k$ from $\mathrm{VAR}$.
The diagram $\delta$ is a
\emph{standard uniform $k$-ary $\tau$-diagram}.
We define $\Vert \delta \Vert^{\M}$ to be the
relation
$\{\, (u_1,...,u_k)\in  M^k\ |\
\M,\frac{(u_1,...,u_k)}{(x_1,...,x_k)}\models \bigwedge \delta\ \}.$
Standard variables are needed in order to uniquely specify
the \emph{order of elements} in tuples of $\Vert \delta \Vert^{\M}$.
Let $\delta$ be a standard uniform $k$-ary $\tau$-diagram.
Let $q\leq k$ be a positive integer.
Let $t:\{1,...,k\}\rightarrow \{1,...,q\}$ be a surjection.
We let $\delta/ t$ denote the set obtained from $\delta$
by replacing each variable $x_i$ by $x_{t(i)}$.
Let $k$ and $q$ be positive integers such that $2\leq q\leq k$.
Let $\eta$ and $\delta$ be standard uniform $q$-ary and $k$-ary $\tau$-diagrams, respectively.
%
%Assume $\tau$ contains symbols of some arity $p\geq k$.
%
Let $f:\{1,...,k\}\rightarrow\{1,...,q\}$
be a surjection. Assume that $\bigwedge \eta \models \bigwedge\delta/f$, i.e.,
the implication $\M,h\models \eta \Rightarrow\ \M,h\models \delta/f$
holds for each $\tau$-model $\M$
and each assignment $h$ interpreting the variables $x_1,...,x_q$ in the domain of $\M$.
Then we write $\eta\, \leq_{f}\, \delta$.
We then define a modal logic
that provides an essentially variable-free
representation of\ \ $\mathrm{UF}_1$.
Define the set $\mathrm{MUF}_1(\tau)$
to be the smallest set $\mathcal{F}$ such that the
following conditions are satisfied.
\begin{enumerate}
%
%\item
%$\bot,\top\in \mathcal{F}$.
%
\item
If $S\in\tau$ is a relation symbol of any arity, then $S\in\mathcal{F}$.
Also $\bot,\top\in \mathcal{F}$.
\item
If $\varphi\in\mathcal{F}$, then $\neg\varphi\in\mathcal{F}$.
If $\varphi_1,\varphi_2\in\mathcal{F}$, then $(\varphi_1\wedge\varphi_2)\in\mathcal{F}$.
\item
If $\varphi_1,...,\varphi_k\in\mathcal{F}$ and $\delta$ is a standard uniform $k$-ary $\tau$-diagram,
then $\langle\delta\rangle(\varphi_1,...,\varphi_k)\in\mathcal{F}$.
\item
If $\varphi\in\mathcal{F}$, then $\langle E\rangle \varphi\in\mathcal{F}$.
(Here $\langle E\rangle$ denotes the \emph{universal modality}; see below for the its semantics.)
\end{enumerate}
The semantics of $\mathrm{MUF}_1(\tau)$ is defined
with respect to \emph{pointed $\sigma$-models} $(\M,w)$, where $\M$ is an ordinary $\sigma$-model
of predicate logic for some vocabulary $\sigma\supseteq \tau$, and $w$ is an
element of the domain $M$ of $\M$.
%
%Such pairs are called \emph{pointed $\sigma$-models}.
%
Obviously we define that $(\M,w)\models\top$
always holds, and that $(\M,w)\models\bot$ never holds.
Let $S\in\tau$ be an $n$-ary relation symbol. We
define $(\M,w)\models S\, \Leftrightarrow\, w^n\in S^{\M}$,
%
%where $w^n$ denotes the tuple in $M^n$ that consists of $n$ copies of the element $w$.
%
where $S^{\M}$ is the interpretation of the relation symbol $S$ in the model $\M$.
The Boolean connectives $\neg$ and $\wedge$ have their usual meaning.
%
%we define $(\M,w)\models \neg\varphi\, \Leftrightarrow (\M,w)\not\models\varphi$
%
%and $(\M,w)\models (\varphi_1\wedge\varphi_2)\ \Leftrightarrow \bigl((\M,w)\models\varphi_1\text{ and }
%
%(\M,w)\models\varphi_2\bigr)$.
%
For formulae of the type $\langle\delta\rangle(\varphi_1,...,\varphi_k)$,
we define that
$(\M,w)\models \langle\delta\rangle(\varphi_1,...,\varphi_k)$
if and only if there exists a tuple $(u_1,...,u_k)\in\ \Vert \delta \Vert^{\M}$ such that
$u_1 = w$ and
$(\M,u_i)\models\varphi_i$ for each $i\in\{1,...,k\}$.
For formulae $\langle E\rangle\varphi$,
we define 
$(\M,w)\models\langle E\rangle\varphi$ if and only if
there exists some $u\in M$ such that $(\M,u)\models\varphi$.
%

%
\begin{comment}
\[
\begin{array}{lll}
%
%(\M,w)\models S\ & \Leftrightarrow & w^n\in S^{\M},\\
%
(\M,w)\models \neg\varphi\ & \Leftrightarrow & (\M,w)\not\models\varphi,\\
%
(\M,w)\models (\varphi_1\wedge\varphi_2)\ & \Leftrightarrow & (\M,w)\models\varphi_1\text{ and }
%
                                                                                                        (\M,w)\models\varphi_2,\\
%
(\M,w)\models \langle\delta\rangle(\varphi_1,...,\varphi_k)\ & \Leftrightarrow &
%
                                                                  \text{there is a tuple }(u_1,...,u_k)\in\ \Vert \delta \Vert^{\M}\text{ such}\\
%
                               & &\text{that } u_1 = w\text{ and }(\M,u_i)\models\varphi_i\text{ for each }i,\\
%
(\M,w)\models\langle E\rangle\varphi\ & \Leftrightarrow & \text{there is some }u\in M\text{ such that }
%
                                                                                                                                           (\M,u)\models\varphi.
%
\end{array}
\]
%
\end{comment}
%
%
%

%
When $\varphi$ is a $\mathrm{MUF}_1(\tau)$-formula and $\M$ a $\sigma$-model
with the domain $M$,
we let $\Vert \varphi \Vert^{\M}$ denote the set
$\{\, u\in M\ |\ (\M,u)\models\varphi\ \}.$
We let $\mathrm{MUF}_1$ denote the union of all sets $\mathrm{MUF}_1(\tau)$,
where $\tau$ is a finite subset of $\mathcal{T}$.
It is very easy to show that there is an effective translation
that turns any formula $\gamma(x)\in\mathrm{DUF}_1$
into a formula $\chi\in\mathrm{MUF}_1$ such that 
$(\M,w)\models\chi\ \Leftrightarrow\ \M,\frac{w}{x}\models\gamma(x)$
for all $\tau$-models $\M$, where $\tau$ is the set of
non-logical symbols in $\gamma(x)$.
(The set of non-logical symbols in 
$\chi$ is contained in $\tau$, and
the formula $\gamma(x)$ can 
either be a sentence or have the free variable $x$.)
%
%Minor non-trivial issues arise in the translation
%
%for example from the fact that $\gamma(x)$ may have 
%
%subformulae of the type $\alpha(y)\wedge\beta(z)$, where $y$ and
%
%$z$ are different variables. Such issues can be very easily dealt with, however.
%
%We leave the details to the journal version of this article.)
%
%The claim is proved in Appendix \ref{uftoduf} (Lemma \ref{duftomuflemma}).
% 

%
%Let $R$ be a relation symbol that occurs as an atom in $\gamma$.
%
%We obtain $\langle E\rangle R\ \leftrightarrow\ $
%

%
%%%%%%%%%%%%%%%%%%%%%%%%%%%%%%%
%___________________Proofs______________________
%%%%%%%%%%%%%%%%%%%%%%%%%%%%%%%%%
%
\section{$\mathrm{UF}_1$ is decidable}\label{decidabilitysection}
Let us \emph{fix} a formula $\psi$ of $\mathrm{MUF}_1$.
We will first define a
translation of $\psi$ to an $\mathrm{MFO}$-formula $\psi^{*}(x)$
in Section \ref{translationsection}.
We will then show in
Sections \ref{firstdirection} and \ref{seconddirection}
that the translation indeed preserves equivalence
of satisfiability over finite models as well as over all models.
Due to the above effective translations from $\mathrm{UF}_1$ to $\mathrm{DUF}_1$
and from $\mathrm{DUF}_1$ to $\mathrm{MUF}_1$,
this implies that the satisfiability and finite
satisfiability problems of $\mathrm{UF}_1$
are decidable.
\subsection{Translating $\mathrm{MUF}_1$ into monadic first-order logic}\label{translationsection}
%

%
%The symbol $\psi$ (and the related formula) 
%
%will remain fixed for all of the rest of the article.
%

%
%If $\psi$ does not contain any formula
%
%subformula of the type $\langle \delta \rangle(\chi_1,...,\chi_k)$,
%
%translation of $\psi$ to an equisatisfiable formula
%
%of monadic first-order logic is trivial.
%
We assume, w.l.o.g.,, that $\psi$ contains
at least one subformula of the type $\langle \delta \rangle(\chi_1,\chi_2)$.
If not, we redefine $\psi$.
%
%and make sure a subformula $\langle \delta \rangle(\chi_1,\chi_k)\vee
%
%\neg \langle \delta \rangle(\chi_1,\chi_2)$ occurs in the resulting formula.
%
The vocabulary of $\psi$ may of course grow.
We also assume, w.l.o.g.,
that $\psi$ does not contain occurrences of
the symbols $\top$, $\bot$.
%
%these may be replaced
%
%by formulae $P\vee\neg P$ and $P\wedge\neg P$,
%
%where $P$ is a relation symbol.
%
Furthermore, we assume, w.l.o.g., that
if $R$ is a relation symbol occurring in
some diagram of $\psi$,
then $\neg R$ also occurs in $\psi$ as a subformula:
we can of course always add the conjunct $R\vee\neg R$ to $\psi$.
Let $V_{\psi}$ be the set of
all relation symbols in $\psi$, whether they
occur in diagrams or as atomic subformulae;
in fact, due to our assumptions above,
the set of atomic formulae in $\psi$
is equal to $V_{\psi}$. 
%
%Let $A_{\psi}$ be the set of \emph{atomic}
%
%formulae that occur in $\psi$,
%
Let $D_{\psi}$ be the  set of relation
symbols occurring in the diagrams of $\psi$.
Let $V_{\psi}(k)$ denote the set of $k$-ary
relation symbols in $V_{\psi}$.
Define $D_{\psi}(k)$ analogously.
Due to the assumption that $\psi$ contains a
subformula $\langle \delta\rangle(\chi_1,\chi_2)$,
each relation symbol of some arity $m\geq 2$ that occurs as an 
atom in $\psi$, also occurs in the diagram $\delta$.
(This is due to the definition of $\mathrm{MUF}_1$.)
%
%Notice that in addition to containing all unary relation symbols in $\psi$,
%
%the set $A_{\psi}$ may contain relation symbols of higher arities.
%
Thus $V_{\psi}(n) = D_{\psi}(n)$ for all $n > 1$.
Let $\mathcal{M}$ denote the maximum arity of all diagrams in $\psi$.
For each $k\in \{2,...,\mathcal{M}\}$, let $\Delta_k$
denote the set of exactly all standard uniform $k$-ary $V_{\psi}$-diagrams.
Let $\Delta$ denote the union of the sets $\Delta_k$,
where $k\in\{2,...,\mathcal{M}\}$.
Let $\mathcal{N}\, :=\, \mathit{max}\{\ | \Delta_k| \ |\ k\in\{2,...,\mathcal{M}\}\ \}$.
Recall that $T(\mathcal{N},\mathcal{M})$ denotes the domain 
of the $\mathcal{N}$-dimensional hypertorus of the
arity $\mathcal{M}$.
%
%Fix a fresh unary relation symbol $P_t$
%
%for each $t\in T(\mathcal{N},\mathcal{M})$.
%
For each $k\in\{2,...,\mathcal{M}\}$, define an injection
$b_k:\Delta_k \longrightarrow \{\, R_1(\mathcal{N},\mathcal{M}),
...,R_{\mathcal{N}}(\mathcal{N},\mathcal{M})\, \}.$
For a $k$-ary diagram $\delta\in\Delta_k$,  let
$T_{\delta}$ denote the
$k$-ary relation $\bigl(b_k(\delta)\bigr)\upharpoonright k$.
Let $\mathrm{SUB}_{\psi}$ denote the
set of subformulae of the formula $\psi$.
Fix fresh unary relation symbols $P_{\alpha}$ and $P_t$
for each formula $\alpha \in \mathrm{SUB}_{\psi}$
and torus point $t\in T(\mathcal{N},\mathcal{M})$.
The vocabulary of the translation $\psi^{*}(x)$ of $\psi$ will be the set
$\{\, P_{\alpha}\ |\ \alpha\in\mathrm{SUB}_{\psi}\ \}
\cup \{\, P_{t}\ |\ t\in T(\mathcal{N},\mathcal{M})\ \}.$
We let $V^*$ denote this set.
We shall next define a collection of
auxiliary formulae needed in
order to define $\psi^{*}(x)$.
If a pointed model $(\M,u)$
satisfies $\psi$, then $\psi^*(x)$ will be 
satisfied in a larger model; the related model
construction is defined in 
the beginning of Section \ref{firstdirection}.
The predicates of the type $P_{\alpha}$
will be used to encode information about sets $\Vert\alpha\Vert^{\M}$,
while the predicates $P_t$ encode information about the \emph{diagrams} of $\psi$.
The predicates $P_t$ are crucial when defining a $V_{\psi}$-model $\B$
that satisfies $\psi$ based on a $V^*$-model $\A$ of $\psi^*(x)$ in Section \ref{seconddirection}.
%

%
%Notice that $\exists y(Rxyy)$ implies $\exists yz(R xyz)$.
%
%If $Rxyy$ and $Rxyz$ are diagrams, we say that
%
%they are in the same \emph{diagram class}.
%
%Define the formula $U_{\psi}$, also called $\mathit{Uniq}$ below,
%
%to be a formula that states in the torus that each possibly
%
%redundant tuple $(u_1,...,u_k)$ satisfies exactly one diagram
%
%class.
%

%
%
%
%
%
%
%
%
%

%
Let $\delta \in \Delta_k$.
Define $\mathit{PreCons}_{\delta} (x_1,...,x_k)$ to be the formula
$$\bigwedge\limits_{\langle\delta\rangle(\chi_1,...,\chi_k)\ \in\ \mathrm{SUB}_{\psi}}
\Bigl(P_{\chi_1}(x_{1})\wedge...\wedge P_{\chi_k}(x_{k})\notag              
\rightarrow
\ P_{\langle\delta \rangle (\chi_1\, ,\, ...\, ,\, \chi_k)}(x_{1})\, \Bigr).\notag
$$
%
%
%
%Let $S$ be the set of all diagrams in $\eta \in \Delta$
%
%such that $\delta \leq_f^{V_{\psi}} \eta$ for some
%
%surjection $f:\{1,...,p\}\rightarrow\{1,...,k\}$.
%
Let $\Delta(\delta)$ be the set of pairs $(\eta,f)$,
where $\eta\in\Delta$ is a $p$-ary diagram
for some $p\geq k$, and $f:\{1,...,p\}\rightarrow\{1,...,k\}$
is a surjection such that we have $\delta \leq_f \eta$.
The set $\Delta(\delta)$ is the set of \emph{inverse projections
of $\delta$ in $\Delta$}.
Define 
\begin{align*}
\mathit{Cons}_{\delta} (x_{1},...,x_{k})\                :=
\bigwedge\limits_{(\eta,f)\, \in\, \Delta(\delta)}
                        \mathit{PreCons}_{\eta}(x_{f(1)},...,x_{f(p)}).\notag              
\end{align*}
The following formula is the principal formula that encodes information about
diagrams of $\delta$ (cf. Lemma \ref{firstlemmaofwarmup}). 
\begin{align*}
& \mathit{Diag}_{\delta} (x_1,...,x_k) :=                                                     
 &                                         \bigvee_{(t_1\, ,\, ...\, ,\, t_k)\ \in\ T_{\delta}}
                 P_{t_1}(x_1)\wedge...\wedge P_{t_k}(x_k)\,
                                  \wedge\, \mathit{Cons}_{\delta}(x_1,...,x_k).
\end{align*}
Let $+(\delta)$ denote the set of relation symbols $R$
that occur positively in $\delta$, i.e., there
exists some atom $R(y_1,...,y_n)\in \delta$, where $n$ is the arity of $R$.
Let $-(\delta)$ be the relation symbols $R$
that occur negatively in $\delta$, i.e., $\neg R(y_1,...,y_n)\in \delta$
for some atom $R(y_1,...,y_n)$.
The following three formulae encode
information about atomic formulae in $\psi$. Define 
\[
\begin{array}{rcl}
\mathit{Local}_{\delta}(x)\  &:=& \                                                       
\bigwedge\limits_{R\, \in\, +(\delta)} P_{{}_R}(x)\ \wedge\
\bigwedge\limits_{R\, \in\, -(\delta)} P_{{}_{\neg R}}(x),\\
\mathit{LocalDiag}_{\delta}(x)\ &:= &\ \mathit{Local}_{\delta}(x)\
\rightarrow\ \mathit{PreCons}_{\delta}(x,...,x)_k,\\ 
\text{ }\ \ \ \psi_{\mathit{local}}\ &:=&\ \bigwedge\limits_{
%
%k\, \in\, \mathit{Ar}_{\psi},\  
%
\delta\, \in\, \Delta} \forall x\, \mathit{LocalDiag}_{\delta}(x).
\end{array}
\]
The next formula is essential in the construction of a $V_{\psi}$-model
of $\psi$ from a $V^*$-model of $\psi^*(x)$ in Section
\ref{seconddirection}. The two models have the same domain.
The formula states that each tuple can be interpreted to satisfy
\emph{some} diagram $\delta$ such that information concerning
the unary predicates in $V^*$ is consistent with $\delta$. See
the way $\B$ is defined based on $\A$ in Section \ref{seconddirection}
for further details. Define
$$
\psi_{\mathit{total}}\ :=\ \bigwedge\limits_{k\, \in\,
\{2,...,\mathcal{M}\}}\forall x_1...\forall x_k\bigvee\limits_{\delta\, \in\, \Delta_k}
\mathit{Cons}_{\delta}(x_1,...,x_k).$$
Also the following formula is crucial for the definition of $\B$.
%
% in Section \ref{seconddirection}.
%
%
%
$$\psi_{\mathit{uniq}}\ :=\ \bigwedge\limits_{t,\, s\ \in\
T(\mathcal{N},\mathcal{M}),\ t\not= s}\neg \exists x \bigl(\, P_t(x)\, \wedge\, P_s(x)\, \bigr).$$
Let $\neg \alpha$, $(\beta\wedge\gamma)$, $\langle E\rangle\chi$,
and $\langle \delta \rangle(\chi_1,...,\chi_k)$
be formulae in $\mathrm{SUB}_{\psi}$.
The following formulae recursively encode information
concerning subformulae of $\psi$. Define 
\[
\begin{array}{lll}
%
%\psi_{{}_{S(x)}}\ &:=&\  
%
%\forall x \Bigl(P_{{}_{S}}(x)
%
%\rightarrow S(x)\Bigr),\\
%
%
%
%\psi_{{}_{R(x,...,x)}}\ &:=&\  
%
%\forall x \Bigl(P_{{}_{R}}(x)
%
%\leftrightarrow R(x,...,x)_k \Bigr),\\
%
%
%
\psi_{\neg \alpha}\ &:=&\  
\forall x \Bigl(P_{\neg \alpha}(x)
\leftrightarrow \neg P_{\alpha}(x)\Bigr),\\
\psi_{(\beta \wedge \gamma)}\ &:=&\ \forall x \Bigl(P_{(\beta \wedge \gamma)}(x)
\leftrightarrow \bigl(P_{\beta}(x)
\wedge P_{\gamma}(x)\bigr)\Bigr),\\
\psi_{\langle E\rangle\chi}\ &:=&\  
\forall x \Bigl(P_{\langle E\rangle\chi}(x)
\leftrightarrow \exists y P_{\chi}(y)\Bigr),\\
\psi_{\langle \delta \rangle(\chi_1,...,\chi_k)}\ &:=&\ \forall x_1
\Bigl(P_{\langle \delta \rangle(\chi_1,...,\chi_k)}(x_1)\\ 
& &
\empty\hspace{1cm}
\leftrightarrow \exists x_2...x_k\bigl(
\mathit{Diag}_{\delta}(x_1,...,x_k)\\
& &
\empty\hspace{2cm}
\wedge\, P_{\chi_1}(x_1)\wedge...\wedge P_{\chi_k}(x_k)\, \bigr)\, \Bigr).
\end{array}
\]
Let $\psi_{\mathit{sub}}\, :=\,
\bigwedge\limits_{\alpha\, \in\, \mathrm{SUB}_{\psi}}
\psi_{\alpha}.$
Finally, we define
\begin{align}
\psi^*(x)\ :=\ \psi_{\mathit{total}}\ \wedge\ \psi_{\mathit{uniq}}\ \wedge\ \psi_{\mathit{local}}\
\wedge\ \psi_{\mathit{sub}}\
\ \wedge\ P_{\psi}(x).\notag
\end{align}
%
%
%

%
%
%
%
%
%
%
%
%
%
%
%
%
%
%
%
%
%

%
%%%%%%%%%%%%%%%%%%%%%%%%%%%%%%%%%%%%%%%%%
%_________________________first_direction________________
%%%%%%%%%%%%%%%%%%%%%%%%%%%%%%%%%%%%%%%%
%
\subsection{Satisfiability of $\psi$ implies satisfiability of $\psi^*(x)$}\label{firstdirection}
%

%
\begin{comment}
%
As in the previous section
%
let $l$ be the maximum arity of all diagrams in $\psi$.
%
For each $k\in \{2,...,l\}$, let $\Delta_k$
%
denote the set of standard uniform $k$-ary $V_{\psi}$-diagrams.
%
Let $m$ be the maximum number in
%
the set $\{\, | \Delta_k |\ |\ k\in\{2,...,l\}\ \}$.
%
Recall that $\Delta$ is union of all sets $\Delta_k$
%
such that $k\in\{2,...,l\}$.
%
Recall that $T(\mathcal{N},\mathcal{M})$ is the domain
%
of the $m$-dimensional hypertorus of the arity $l$.
%
\end{comment}
%

%
Fix an arbitrary model $V_{\psi}$-model $\M$
with the domain $M$.
Fix a point $w\in M$.
Assume $(\M,w)\models\psi$.
We shall next construct a model $\T$
with the domain $M\times T(\mathcal{N},\mathcal{M})$.
We then show that $\T,\frac{(w,t)}{x}\models\psi^*(x)$,
where $t$ is a torus point.
If $\M$ is a finite model, then so is $\T$.
The domain $M\times T(\mathcal{N},\mathcal{M})$ of the $V^*$-model $\T$ consists of
copies of $M$, one copy for each torus point $t\in T(\mathcal{N},\mathcal{M})$.
Let us define  interpretations of the symbols in $V^*$.
Consider a symbol $P_{\alpha}$, where $\alpha \in \mathrm{SUB}_{\psi}$.
If $(u,t)\in\mathit{Dom}(\T)$, then
$(u,t)\, \in\, P_\alpha^{\T}\ \Leftrightarrow\ 
u\, \in\ \Vert \alpha \Vert^{\M}$.
Consider then a symbol $P_{t}$, where $t \in T(\mathcal{N},\mathcal{M})$.
If $(u,t')\in\mathit{Dom}(\T)$, then
$(u,t')\, \in\, P_{t}^{\T}\ \Leftrightarrow\ t' = t$.
%

%
%%%%%%%%%%%%%%%%%%%%%%%%%%%%%
%____________________lemma_______________
%%%%%%%%%%%%%%%%%%%%%%%%%%
%
\begin{lemma}\label{firstlemmaofwarmup}
Let $\langle \delta \rangle(\chi_1,...,\chi_k) \in \mathrm{SUB}_{\psi}$
and $(u,t) \in Dom(\T)$. Then 
$(\M,u) \models \langle \delta \rangle (\chi_1,...,\chi_k)$
 iff 
$\T,\frac{(u,t)}{x_1} \models
\exists x_2...\exists x_k\bigl(\, \mathit{Diag}_{\delta}(x_1,...,x_k)\,
\wedge\, P_{\chi_1}(x_1) \wedge...\wedge P_{\chi_k}(x_k)\, \bigr)$.
\end{lemma}
\begin{proof}
Define $u_1 := u$ and $t_1 := t$.
Assume $(\M,u_1) \models \langle \delta \rangle (\chi_1,...,\chi_k)$.
Thus $(u_1,...,u_k) \in\ \Vert \delta \Vert^{\M}$ for some
tuple $(u_1,...,u_k)$ such that 
$u_i\, \in\ \Vert \chi_i \Vert^{\M}$
for each $i$. Hence $(u_i,s)\in P_{\chi_i}^{\T}$ for each $i$
and each torus point $s$.
%
%If all the points $u_1,...,u_k$ are equal,
%
%then it is immediate that $\T,\frac{(u_1,t_1)}{x_1}\models \delta(\chi_1,...,\chi_k)(x_1)$.
%
To conclude the first direction of the proof, 
%
%
%
\begin{comment}
$
%
\T,\frac{(u_1,t_1)}{x_1}\models \exists x_2...\exists x_k
%
\bigl(\, \mathit{Diag}_{\delta}(x_1,...,x_k)
%
\wedge P_{\chi_1}(x_1)\, \wedge...\wedge\, P_{\chi_k}(x_k)\bigr),
%
$
\end{comment}
%
%\end{multline*}
%
%\end{comment} 
%
%
%
it suffices to prove that
$\T,\frac{\bigl((u_1,t_1),...,(u_k,t_k)\bigr)}{(x_1,...,x_k)}\models$ $
\mathit{Diag}_{\delta}(x_1,...,x_k)$
for some torus points $t_2,...,t_k$. 
Let $t_2,...,t_k$ be the torus points
such that $(t_1,...,t_k)\in T_{\delta}$.
In order to establish that 
$\T,\frac{\bigl((u_1,t_1),...,(u_k,t_k)\bigr)}{(x_1,...,x_k)}\models
\mathit{Cons}_{\delta}(x_1,...,x_k),$
%
%
%
%it now suffices to establish that
%
%
%
assume that $\delta\, \leq_{f}\, \eta$,
where $\eta\in\Delta_p$ and $p\geq k$.
Assume that $\langle \eta
\rangle(\gamma_1,...,\gamma_p) \in \mathrm{SUB}_{\psi}$,
and that
$\T,\frac{\bigl((u_1,t_1),...,(u_k,t_k)\bigr)}{(x_1,...,x_k)}\models
P_{\gamma_1}(x_{f(1)})\, \wedge...\wedge\, P_{\gamma_p}(x_{f(p)}).$
We must show that
$(u_{f(1)},t_{f(1)})
\ \in\ P_{\langle\eta \rangle(\gamma_1,...,\gamma_p)}^{\T}$.\smallskip
%

%
%Therefore $u_{\sigma^{-1}(1)} \in\ \Vert \langle\delta^{\sigma}\rangle (\gamma_1,...,\gamma_k) \Vert^{\M}$.
%
For each $i\in\{1,...,p\}$, as $(u_{f(i)},t_{f(i)})
\ \in\ P_{\gamma_{i}}^{\T}$, we have $u_{f(i)}\in\Vert\gamma_i\Vert^{\M}$
by the definition of $P_{\gamma_{i}}^{\T}$.
As $(u_1,...,u_k) \in\ \Vert \delta \Vert^{\M}$
and $\delta\, \leq_{f}\, \eta$,
we have $(u_{f(1)},...,u_{f(p)})\in \Vert\eta\Vert^{\M}$.
Therefore we have $u_{f(1)}\in\Vert\langle \eta\rangle(\gamma_1,...,\gamma_p)\Vert^{\M}$.
%
%
%
%$$\ P_{\langle \delta^{\sigma} \rangle(\gamma_1,...,\gamma_k)}^{\T}
%
%\ \ =\ \ 
%
%\Vert \langle\delta^{\sigma}\rangle (\gamma_1,...,\gamma_k)
%
%\Vert^{\M}\ \times\ \ \mathit{Dom}(\T),$$
%
%
%
Thus
$(u_{f(1)},t_{f(1)})
\in P_{\langle \eta \rangle(\gamma_1,...,\gamma_p)}^{\T}$
by the definition of $P_{\langle \eta \rangle(\gamma_1,...,\gamma_p)}^{\T}$.
We then deal with the converse implication of the lemma.
Define $s_1 := t$ and $ v_1 := u$.
Assume 
%
%
%
%\begin{multline*}
%
$
\T,\frac{(v_1,s_1)}{x_1}\ \models\ \exists x_2...\exists x_k 
\bigl(\, \mathit{Diag}_{\delta}(x_1,...,x_k)
\wedge\, P_{\chi_1}(x_1)\, \wedge...\wedge\, P_{\chi_k}(x_k)\, \bigr).
$
%
%\end{multline*}
%
%
%
Hence
%
%
%
%\begin{align*}
%
$\T,\frac{\bigl((v_1,s_1),...,(v_k,s_k)\bigr)}{(x_1,...,x_k)}
\models\mathit{Diag}_{\delta}(x_1,...,x_k)$
%
%\end{align*}
%
%
%
for some tuple $\bigl((v_1,s_1),...,(v_k,s_k)\bigr)$
%
%\in\ \mathit{Dom}(\T)$
%
such that $(v_i,s_i) \in P_{\chi_i}^{\T}$ for each $i$.
As now
%
%
%
%\begin{align*}
%
$$\T,\frac{\bigl((v_1,s_1),...,(v_k,s_k)\bigr)}{(x_1,...,x_k)}
\models\mathit{PreCons}_{\delta}(x_1,...,x_k),$$
%
%\end{align*}
%
%
%
we  infer that $(v_1,s_1) \in P_{\langle \delta \rangle(\chi_1,...,\chi_k)}^{\T}$.
By the definition of $P_{\langle \delta \rangle(\chi_i,...,\chi_k)}^{\T}$,
%
%As by definition
%
%
%
%$$P_{\langle \delta \rangle(\chi_i,...,\chi_k)}^{\T}\ \
%
%=\ \ \Vert \langle \delta\rangle(\chi_1,...,\chi_k)
%
%\Vert^{\M}\times\ \mathit{Dom}(\T),$$
%
%
%
we have $(\M,v_1)\models 
\langle\delta\rangle (\chi_1,...,\chi_k)$.
%
%as desired.
%
\end{proof}
\begin{lemma}\label{secondlemmaofwarmup}
Let $t$ be any torus point.
Under the assumption $(\M,w)\models
\psi$, we have $\T,\frac{(w,t)}{x}
\models \psi^{*}(x)$.
\end{lemma}
\begin{proof}
See the appendix.
\end{proof}
%

%
%_________________________second_direction__________________________
%

%
\subsection{Satisfiability of $\psi^*(x)$ implies satisfiability of $\psi$}\label{seconddirection}
Let $\A$ be a $V^*$-model with the domain $A$.
Assume that $\A,\frac{w}{x} \models \psi^{*}(x)$.
We next define a $V_{\psi}$-model $\B$ with
the same domain $A$, and then show that $(\B,w)\models\psi$.
Let $U$ be a non-empty set, and let $p\in\mathbb{Z}_+$.
Let $(u_1,...,u_p)\in U^p$ be a tuple.
We say that the tuple $(u_1,...,u_p)$ \emph{spans
the set $\{u_1,...,u_p\}$}.
Let $k\in\mathbb{Z}_+$, and let $S\in V_{\psi}$ be a $k$-ary symbol.
We define $(u,...,u)_k\in S^{\B}$ iff
$u\in P_{{}_S}^{{}\, \A}$.
%
%On tuples $(u_1,...,u_k)\in A^k$
%
%
%
This settles the interpretation of the symbols $S\in V_{\psi}$
on tuples that span sets of size one. Interpretation
of the symbols on tuples that span larger sets
is more complicated. We begin with the following lemma.
\begin{lemma}\label{uniquelemma}
Let $u_1,...,u_k\in A$.
Assume $\A,\frac{(u_1,...,u_k)}{(x_1,...,x_k)}\models\mathit{Diag}_{\delta}(x_1,...,x_k)$.
Then $\A,\frac{(u_{\sigma(1)},...,u_{\sigma(k)})}{(x_1,...,x_k)}\not
\models\mathit{Diag}_{\eta}(x_1,...,x_k)$
holds for all all $k$-permutations $\sigma$
and all $\eta\in\Delta_k\setminus\{\delta\}$.
Also $\A,\frac{(u_{\mu(1)},...,u_{\mu(k)})}{(x_1,...,x_k)}\not
\models\mathit{Diag}_{\delta}(x_1,...,x_k)$ holds
for all $k$-permutations $\mu$ other than the identity permutation.
\end{lemma}
\begin{proof}
Straightforward by Lemma \ref{toruslemma}.
%
%(A full proof is given in Appendix \ref{v3appendix}.)
%
\end{proof}

%
\begin{comment}
%
Let $p$ and $p'$ be positive integers, and let $K$ be a set.
%
Let $(s_1,...,s_p)\in K^p$ and $(v_1,...,v_{p'})\in K^{p'}$
%
be tuples. We write $(s_1,...,s_p) \equiv (v_1,...,v_{p'})$ if
%
$\{s_1,...,s_p\} = \{v_1,...,v_{p'}\}$.
%
Define $\mathit{set}\bigl((s_1,...,s_p)\bigr)\, :=\, \{s_1,...,s_p\}$.
%
\end{comment}
%

%
Let $q\in\{2,...,\mathcal{M}\}$.
Consider subsets of $A$ that have exactly $q\geq 2$ elements.
Let us divide such sets into two classes.
Let $U = \{u_1,...,u_q\}$ be a set with $q$
distinct elements. 
Assume first that there exists some
$q$-permutation $\sigma$ and some $\eta\in\Delta_q$
such that
$\A,\frac{(u_{\sigma(1)},...,u_{\sigma(q)})}{(x_1,...,x_q)}
\models\mathit{Diag}_{\eta}(x_1,...,x_q).$
Define $\mathit{tuple}( U ) := (u_{\sigma(1)},...,u_{\sigma(q)})$
and $\mathit{diagram}( U ) := \eta$.
Define also $\mathit{type}(U) = 1$.
Assume then that
$\A,\frac{(u_{\sigma(1)},...,u_{\sigma(q)})}{(x_1,...,x_q)}
\not\models\mathit{Diag}_{\eta}(x_1,...,x_q)$
holds for all $\eta\in\Delta_q$ and all $q$-permutations $\sigma$.
As $\A\models\psi_{\mathit{total}}$,
there exists some diagram $\delta\in\Delta_q$
such that $\A,\frac{(u_1,...,u_q)}{(x_1,...x_q)}
\models \mathit{Cons}_{\delta}(x_1,...,x_q)$.
Define $\mathit{tuple}( U ) = (u_1,...,u_q)$
and $\mathit{diagram}( U ) := \delta$.
Define also $\mathit{type}(U) = 2$.
%

%
%Thus for each $p \in\mathcal{A}$ and 
%
%each $(v_1,...,v_p) = \mathcal{U}$, the tuple
%
%$D(u_1,...,u_k)$ satisfies has the property
%
%that $\A,\frac{(u_1,...,u_p)}{(x_1,..,x_p)}
%
%\models\mathit{Cons}_{D(u_1,...,u_k)}(x_1,...,x_p)$.
%

%
Notice that by our assumptions
in Section \ref{translationsection},
there are no relation symbols 
$S\in V_{\psi}\setminus D_{\psi}$
of any arity higher than one.
%
%On tuples that span a set of some size in $\{2,...,\mathcal{M}\}$,
%
%the definition of the relations $S^{\B}$ is
%
%somewhat more complicated.
%
Recall that $\mathcal{M}$ is the maximum arity of diagrams in $\Delta$.
We next define the relations $S^{\B}$, where $S\in D_{\psi}$, on tuples
of elements of $A$ that span sets with $q\in\{2,...,\mathcal{M}\}$ elements.
The definition has the property---as Lemma \ref{lemmawhatever} below 
establishes---that if $(u_1,...,u_k)\, \in\ \Vert \delta \Vert^{\B}$,
where $\delta\in\Delta_k$, then 
$\A,\frac{(u_1,...,u_k)}{(x_1,...,x_k)}\models \mathit{PreCons}_{\delta}(x_1,...,x_k)$.
In fact this holds also for tuples that span a singleton set,
see Lemma \ref{lemmawhatever}.
Let $q\in\{2,...,\mathcal{M}\}$,
and let $U \subseteq A$
be a set of the size $q$.
Assume first that $\mathit{type}(U) = 1$.
Let $\mathit{diagram}(U) = \eta\in\Delta_q$ and 
$\mathit{tuple}(U) = (u_1,...,u_q)$.
We have 
$\A,\frac{(u_1,...,u_q)}{(x_1,...,x_q)}\models
\mathit{Diag}_{\eta}(x_1,...,x_q).$
Let $k\geq q$ be an integer.
Interpret each $k$-ary symbol $S\in D_{\psi}$ such that
$\B,\frac{(u_1,...,u_q)}{(x_1,...,x_q)}\models \eta.$
This definition uniquely specifies the interpretation
of $S$ on each $k$-ary tuple that spans the set $\{u_1,...,u_q\}$.
To see this, let $f:\{1,...,k\}\rightarrow\{1,...,q\}$ be a surjection.
Now we have $(u_{f(1)},...,u_{f(k)})\in S^{\B}$ iff
$S(x_{f(1)},...,x_{f(k)})\in \eta$.
Assume then that $\mathit{type}(U) = 2$.
Let $\mathit{diagram}(U) = \delta\in\Delta_q$
and $\mathit{tuple}(U) = (v_1,...,v_q)$.
We have
$\A,\frac{(v_1,...,v_q)}{(x_1,...,x_q)}\models
\mathit{Cons}_{\delta}(x_1,...,x_p).$
Let $k\geq q$ be an integer.
Interpret each $k$-ary symbol $S\in D_{\psi}$ such that
$\B,\frac{(v_1,...,v_q)}{(x_1,...,x_q)}\models \delta.$
We investigate each $q\in\{2,...,\mathcal{M}\}$,
and thereby obtain a complete definition of $\B$;
if there are symbols of some arity
$r>\mathcal{M}$ in $D_{\psi}$,
we arbitrarily define the interpretations of such symbols on tuples
that span sets with more than $\mathcal{M}$ elements.
%

%%%%%%%%%%%%%%%%%%%%%%%%%%%%%%%%%

%%%%%%%%%%%%%%%%%%%%

%
\begin{lemma}\label{lemmasomething}
If $(u_1,...,u_k)\, \in\ A^k$ and
$\A,\frac{(u_1,...,u_k)}{(x_1,...,x_k)}\models \mathit{Diag}_{\delta}(x_1,...,x_k)$
for some $\delta\in\Delta_k$,
then $(u_1,...,u_k)\, \in\ \Vert \delta \Vert^{\B}$. 
\end{lemma}
\begin{proof}
Assume $\A,\frac{(u_1,...,u_k)}{(x_1,...,x_k)}\models \mathit{Diag}_{\delta}(x_1,...,x_k)$.
Notice that $k\geq 2$, since diagrams have by definition an arity at least two.
As $\A\models\psi_{\mathit{uniq}}$,
the set $U = \{u_1,...,u_k\}$ has exactly $k$ elements.
We have $\mathit{type}(U) = 1$,
and by Lemma \ref{uniquelemma},
$\mathit{tuple}(U) = (u_1,...,u_k)$.
Thus $\B,\frac{(u_1,...,u_k)}{(x_1,...,x_k)}\models \delta$.
\end{proof}
%

%
%Recall that $\mathit{Ar}_{\psi}$ denotes the 
%
%set of arities of diagrams that occur in $\psi$.
%
%Recall that $\mathcal{M}$ is the
%
%maximum arity of the relation symbols 
%
%that occur in the diagrams of $\delta$.
%

%
\begin{lemma}\label{lemmawhatever}
Let $k\in\{1,...,\mathcal{M}\}$.
If $(u_1,...,u_k)\, \in\ \Vert \delta \Vert^{\B}$, 
where $\delta\in\Delta_k$, then 
$\A,\frac{(u_1,...,u_k)}{(x_1,...,x_k)}\models \mathit{PreCons}_{\delta}(x_1,...,x_k)$.
\end{lemma}
\begin{proof}
The case where $(u_1,..,u_k)$ spans a singleton set
follows since $\A\models \psi_{\mathit{local}}$.
Let us consider the cases where $(u_1,..,u_k)$ spans a
set of the size two or larger.
% 
\begin{comment}
%
Let $R\in A_{\psi}\, \cap\, D_{\psi}$ be an $n$-ary symbol. Assume that $(\B,u)\models R$.
%
Therefore the tuple $(u,...,u)_n$ is in $R^{\B}$.
%
As $R\in D_{\psi}$, there is some $k\in\mathit{Ar}_{\psi}$
%
such that $k\leq n$, $\delta\in\Delta_k$,  and
%
$(u,...,u)\in\ \Vert \delta \Vert^{\B}$.
%
As $\A\models\psi_{\mathit{total}}$,
%
we have $\A,\frac{u}{x}\models \mathit{Cons}_{\delta}(x,...,x)_k$.
%
Thus $\A,\frac{u}{x}\models\mathit{Local}_R(x)$.
%
As $\A\models\psi_{\mathit{sub}}$,
%
we conclude that $\A,\frac{u}{x}\models P_{{}_R}(x)$.
%
\end{comment}
%
%Let $k\in\{2,...,\mathcal{M}\}$.
%
%Let $A^{(k)}$ be the set of tuples in $A^k$
%
%that span a set of the size two or larger.
%

%
%For each tuple $(u_1,...,u_k)\in A^{k}$,
%
%there exists exactly one $\delta\in\Delta_k$
%
%such that $(u_1,...,u_k)\, \in\ \Vert \delta \Vert^{\B}$.
%
%We will show that if 
%
%$(u_1,...,u_k)\, \in\ \Vert \delta \Vert^{\B}$, then
%
%$\A,\frac{(u_1,...,u_k)}{(x_1,...,x_k)}\models \mathit{Cons}_{\delta}(x_1,...,x_k)$.
%

%
Assume that $(u_1,...,u_k)\in \Vert\delta\Vert^{\B}$ is a tuple
such that $U = \{u_1,...,u_k\}$ contains exactly $q\geq 2$
elements.
%
\begin{comment}
%
Let $(u_1,...,u_k)\in \Vert\delta\Vert^{\B}$ be a tuple
%
such that $|\{u_1,...,u_k\}| = q \geq 2$.
%
\end{comment}
%
Let  $m:\{1,...,q\}\rightarrow\{1,...,k\}$
be an injection such that
the tuple $(u_{m(1)},...,u_{m(q)})$
spans the set $\{u_1,...,u_k\}$.
%
%and if $m(i) = m(j)$, then $u_{m(i)} = u_{m(j)}$.
%

%
Assume first that we have
$\A,\frac{(u_{m(\sigma (1))},...,u_{m(\sigma(q))})}{(x_1,...,x_q)}
\models \mathit{Diag}_{\eta}(x_1,...,x_q)$
for some $\eta\in \Delta_q$ and some $q$-permutation $\sigma$.
Thus $\mathit{type}(U) = 1$.
By Lemma \ref{uniquelemma}, we 
have $\mathit{tuple}(U) = (u_{m (\sigma(1))},...,u_{m (\sigma(q))})$
and $\mathit{diagram}(U) = \eta$.
Let $s:\{1,...,q\}\rightarrow\{1,...,k\}$ be the injection
such that $s(i) = m(\sigma(i))$ for each $i\in \{1,...,q\}$.
As $\mathit{tuple}(U) = (u_{s(1)},...,u_{s(q)})$,
we have 
$\B,\frac{(u_{s(1)},...,u_{s(q)})}{(x_1,...,x_q)}\models \eta.$
As $\A,\frac{(u_{s (1)},...,u_{s(q)})}{(x_1,...,x_q)}
\models \mathit{Diag}_{\eta}(x_1,...,x_q)$,
we have 
$$\A,\frac{(u_{s(1)},...,u_{s(q)})}{(x_1,...,x_q)}
\models \mathit{Cons}_{\eta}(x_1,...,x_q).$$
The rest or the argument for the case where $\mathit{type}(U) = 1$,
will be dealt with below. Let us next elaborate some details related to the case
where $\mathit{type}(U) = 2$.
So, assume $\mathit{type}(U) = 2$. 
Let $t:\{1,...,q\}\rightarrow\{1,...,k\}$
be an injection such that 
$\mathit{tuple}(U) = (u_{t(1)},...,u_{t(q)})$.
Let $\mathit{diagram}(U) = \rho \in \Delta_q$.
Thus $\A,\frac{(u_{t(1)},...,u_{t(q)})}{(x_1,...,x_q)}\models
\mathit{Cons}_{\rho}(x_1,...,x_q)$
and $\B,\frac{(u_{t(1)},...,u_{t(q)})}{(x_1,...,x_q)}\models \rho$.
We then complete the arguments for both cases $\mathit{type}(U) = 1$
and $\mathit{type}(U) = 2$.
Let $(h,\nu) \in \{(s,\eta),(t,\rho)\}$,
where $s$ and $t$ are the injections defined above,
and of course $\eta$ and $\rho$ are the related diagrams. 
Let $g:\{1,...,k\}\rightarrow\{1,...,q\}$ be the surjection such that
$g(i) = j$ iff $u_i = u_{h(j)}$.
Notice that $(u_{ h(1)},...,u_{ h(q)}) \in \Vert \nu \Vert^{\B}$
and $(u_1,...,u_k) \in \Vert \delta \Vert^{\B}$,
and these two tuples span the same set with $q$ elements.
Thus we have $\nu \leq_g \delta$.
%

%
%As $\A,\frac{(u_{h(1)},...,u_{h(q)})}{(x_1,...,x_q)}
%
%\models \mathit{Diag}_{\eta}(x_1,...,x_q)$,
%
We have
$\A,\frac{(u_{h(1)},...,u_{h(q)})}{(x_1,...,x_q)}\models
\mathit{Cons}_{\nu}(x_1,...,x_q).$
As $\nu \leq_g \delta$, we have
$$\A,\frac{(u_{h(1)},...,u_{h(q)})}{(x_1,...,x_q)}\models
\mathit{PreCons}_{\delta}(x_{g(1)},...,x_{g(k)}).$$
Recalling that $g(i) = j$ iff $u_i = u_{h(j)}$, we conclude that
$\A,\frac{(u_1,...,u_k)}{(x_{1},...,x_{k})}\models
\mathit{PreCons}_{\delta}(x_{1},...,x_{k}),$
as required.
\end{proof}
\begin{lemma}\label{thirdlemmaofwarmup}
Let $\alpha \in \mathrm{SUB}_{\psi}$ and $u\in A$.
We have $(\B,u)\models\alpha$ iff
$\A,\frac{u}{x} \models P_{\alpha}(x)$.
\end{lemma}
\begin{proof}
See the appendix.
\end{proof}

Due to Lemma \ref{thirdlemmaofwarmup}, we immediately observe that
since $\A, \frac{w}{x} \models P_{\psi}(x)$, we must have $(\B,w) \models \psi$.
Together with Lemma \ref{secondlemmaofwarmup}, 
this establishes the following theorem.
\begin{theorem}
The satisfiability and finite satisfiability problems of
the one dimensional fragment are decidable.
\end{theorem}
%

%
\begin{comment}
%
A first-order formula $\varphi$ is satisfiable in some model 
%
iff the existential second-order formula $\exists\overline{R}\varphi$,
%
where $\exists\overline{R}$ quantifies all the non-logical symbols in $\varphi$,
%
is satisfiable in some model with the empty vocabulary.
%
We can define a generalization of the satisfiability problem by
%
allowing as inputs formulae that begin with an block
%
of alternating monadic second-order quantifiers, and the block
%
is followed by a an existential second-order sentence that 
%
quantifies away all relation symbols in the first-order
%
part of the formula, except the monadic variables
%
that the monadic second-order block binds; the
%
question is whether the input formula is satisfied in a model
%
with the empty vocabulary. Notice that since monadic second-order logic is
%
decidable over models with the empty vocabulary, we obtain that
%
the generalized decision problem for $\mathrm{UF}_1$  is decidable:
%
as an input we have a formula beginning with a block of alternating
%
monadic second-order quantifiers,
%
and this is followed by an existential second-order sentence, whose
%
first-order part is in $\mathrm{UF}_1$. The problem whether
%
such an expression is satisfied in a model with the empty vocabulary,
%
is  decidable; by an easy modification
%
of our proof, one sees that the input translates to an equisatisfiable sentence of
%
monadic second-order logic.
%
\end{comment}
%

%
\section{Undecidable extensions}\label{undecidability}
%

%
%%%%%%%%%%%%%%%%%%%%%%%%%%%%%%%
%___________________Tilings____________
%%%%%%%%%%%%%%%%%%%%%%%%%%%%%%
%
\newcommand{\LL}{\mathcal{L}}
\newcommand{\G}{\mathfrak{G}}
\newcommand{\NN}{\mathbb{N}}
\newcommand{\cT}{\mathbb{T}}
\newcommand{\TDUF}{\mathrm{UF}_3}
\newcommand{\ODNF}{\mathcal{NF}_1}

%
%\section{Undecidability for stronger fragments}
%

The \emph{general one-dimensional fragment} $\mathrm{GF}_1$ of first-order logic
is defined in the same way as $\mathrm{UF}_1$, except that the uniformity
condition is relaxed.  
The set of $\tau$-formulae of $\mathrm{GF}_1$ is
the smallest set $\mathcal{F}$ satisfying the following conditions.
\begin{enumerate}
%
%\item
%$\bot,\top\in \mathcal{F}$
%
\item
If $\varphi$ is a unary $\tau$-atom, then $\varphi\in\mathcal{F}$.
Also $\top,\bot\in\mathcal{F}$.
\item
If $\varphi\in \mathcal{F}$, then $\neg\varphi\in\mathcal{F}$.
If $\varphi_1,\varphi_2\in \mathcal{F}$,
then $(\varphi_1\wedge\varphi_2)\in \mathcal{F}$.
\item
Let $Y = \{y_1,...,y_k\}$ be a set of variable symbols.
Let $U$ be a finite set of
formulae $\psi\in\mathcal{F}$ with free variables in $Y$.
Let $F$ be a set of $\tau$-atoms with free variables in $Y$.
Let $\varphi$ be any Boolean combination of formulae in $F\cup U$.
Then $\exists y_2...\exists y_k\, \varphi\, \in \mathcal{F}$
and $\exists y_1...\exists y_k\, \varphi\, \in \mathcal{F}$.
\end{enumerate}

There are different natural ways of generalizing $\mathrm{UF}_1$ so that 
a two-dimensional logic is obtained.
%
%One way would be to apply the uniformity
%
%condition only to atoms of arities higher than two, and allow free Boolean combinations
%
%of formulae with two free variables.
%
Here we consider a formalism which we call the 
\emph{strongly uniform two-dimensional fragment} $\mathrm{SUF}_2$ of first-order logic.
The set of $\tau$-formulae of $\mathrm{SUF}_2$ is
the smallest set $\mathcal{F}$ satisfying the following conditions.
\begin{enumerate}
%
%\item
%$\bot,\top\in \mathcal{F}$
%
\item
If $\varphi$ is a unary or a binary $\tau$-atom, then $\varphi\in\mathcal{F}$.
Also $\top,\bot\in\mathcal{F}$.
\item
If $\varphi\in \mathcal{F}$, then $\neg\varphi\in\mathcal{F}$.
If $\varphi_1,\varphi_2\in \mathcal{F}$,
then $(\varphi_1\wedge\varphi_2)\in \mathcal{F}$.
\item
Let $y_1$ and $y_2$ be variable symbols.
Let $U$ be a finite set of
formulae $\psi\in\mathcal{F}$ whose free variables are in $\{y_1,y_2\}$.
Let $\varphi$ be any Boolean combination of formulae in $U$.
Then $\exists y_2\, \varphi\, \in \mathcal{F}$
and $\exists y_1\exists y_2\, \varphi\, \in \mathcal{F}$.
\item
Let $Y = \{y_1,...,y_k\}$, $k\ge 3$, be a set of variable symbols.
Let $U$ be a finite set of
formulae $\psi\in\mathcal{F}$ such that each $\psi$ has at most 
one free variable, and the variable is in $Y$.
Let $F$ be a $V$-uniform set, $V\subseteq Y$, of $\tau$-atoms.
Let $\varphi$ be any Boolean combination of formulae in $F\cup U$.
Then $\exists y_3...\exists y_k\, \varphi\, \in \mathcal{F}$,
$\exists y_2...\exists y_k\, \varphi\, \in \mathcal{F}$ 
and $\exists y_1...\exists y_k\, \varphi\, \in \mathcal{F}$.
\end{enumerate}
Both of these extensions of 
$\mathrm{UF}_1$ are $\Pi_1^0$-complete;
see the appendix for the proofs.
This shows that if we lift either of the two
principal syntactic restrictions of $\mathrm{UF}_1$,
we obtain an undecidable formalism.
\section{Expressivity}\label{expressivity}

\emph{Guarded negation first-order logic} $\mathrm{GNFO}$ is a
novel fragment of first-order logic introduced
in \cite{IEEEonedimensional:barany}.
$\mathrm{GNFO}$ subsumes the
guarded fragment $\mathrm{GFO}$.
It turns out that $\mathrm{UF}_1$ is incomparable
in expressivity with both $\mathrm{GNFO}$ and
the two-variable fragment with counting quantifiers $\mathrm{FOC}^2$.
This is proved in the appendix.
% 

%%%%%%%%%%%%%%%%%%%%%%%%55
%%%%%%%%%%%%%%%%%%%%%%%%%%%%%%%
%%%%%%%%%%%%%%%%%%%%%%%%%%%%
%%%%%%Oma teksti loppuu
%%%%%%%%%%%%%%%%%%
%%%%%%%%%%%%%%%%%%%%%%%%
%%%%%%%%%%%%%%%%%%%%%%5
%%%%%%%%%%%%%%%%%%%%%%%5
%%%%%%%%%%%%%%%%%%%%%

%
\section{Conclusion}

The principal contribution of this paper is the
discovery of the fragment $\mathrm{UF}_1$
via the introduction of the notions of 
\emph{uniformity} and \emph{one-dimensionality}.
%
%The logic $\mathrm{UF}_1$
%
%of first-order logic that extends equality-free
%
%two-variable logic in a canonical way.
%
%We have discovered that the notions of
%
%\emph{one-dimensionality} and \emph{uniformity}
%
%lead to the definition of the decidable fragment $\mathrm{UF}_1$ of first-order logic.
%
The notions offer a novel perspective on why
modal logics are robustly decidable.
Also, $\mathrm{UF}_1$   
extends equality-free $\mathrm{FO}^2$ in a natural way,
and thus provides a possible
novel direction in the currently
very active research on two-variable logics.
Additionally, we believe that our satisfiability preserving translation of $\mathrm{UF}_1$
into the monadic class is of
an independent mathematical interest.
The translation is clearly robust and can be altered
and extended to give
other decidability (and complexity) proofs.
In the future we intend to study
variants of $\mathrm{UF}_1$ with identity.
It was observed in \cite{IEEEonedimensional:barany} that
adding the formula
$\forall x \forall y\bigl(Rxy \leftrightarrow x \not= y\bigr)$
to $\mathrm{GNFO}$
leads to an undecidable formalism.
It is not immediately
clear whether the extension of $\mathrm{UF}_1$ with 
the free use of equality and inequality results in undecidability.
%
%
%
%Of course $\mathrm{UF}_1$ already contains
%
%some natural restricted uses of identity. For
%
%example $\exists x\exists y(Rxy\wedge x = y)$
%
%can be expressed simply by the formula $\exists x (Rxx)$.
%
%On the other hand, the formula $\forall x \forall y\bigl(Rxy \leftrightarrow x \not= y\bigr)$,
%
%for instance, is not
%
%expressible in $\mathrm{UF}_1$.
%
%This can be seen (for example) by our decidability
%
%proof, which entails that satisfiable $\mathrm{UF}_1$-formulae can always be
%
%satisfied in larger models.  
%
%Interestingly, the formula $\forall x \forall y\bigl(R(x,y) \leftrightarrow x \not= y\bigr)$ is
%
%also inexpressible in $\mathrm{GNFO}$, as was observed in the conclusion of
%
%\cite{IEEEonedimensional:barany}. Furthermore, it was observed
%
%that adding this kind of a use of equality to $\mathrm{GNFO}$ results in an
%
%undecidable logic. It remains to be seen what kinds of uses of identity can be
%
%added to $\mathrm{UF}_1$ without rendering the logic undecidable.
%
%Also, it will be interesting to see how far and to which direction
%
%our proof method can be pushed.
%
We are currently working on related decidability and
complexity questions.
%

%
% conference papers do not normally have an appendix
%

\bibliographystyle{plain}
\bibliography{arxivi}

%% Appendix.
%% Remove the \Appendix command if an 
%% appendix is not required.
\appendix

\section{Translation\, $\mathrm{UF}_1\rightarrow\mathrm{DUF}_1$ }\label{uftoduf}
\begin{proposition}\label{uftoduflemma}
There is an effective translation that transforms
each formula in $\mathrm{UF}_1$ to 
an equivalent formula in $\mathrm{DUF}_1$.
\end{proposition}
\begin{proof}
Let $\chi := \exists y_2... \exists y_k\, \varphi$ be a
formula of $\mathrm{UF}_1$ formed using the formation rule (iii)
in the definition of $\mathrm{UF}_1$. We may assume, w.l.o.g.,
that that the
variables $y_1,...,y_k$ are distinct, and that $k\geq 2$.
Define $Y := \{y_1,...,y_k\}$.
%
%and $X := Y\setminus \{y_1\}$.
%
Let $\tau_{\chi}$ be the set of relation symbols in $\chi$ of the arity two and higher.
Put $\varphi$ into disjunctive normal form.
We obtain a formula $\exists y_2... \exists y_k\, \bigl(\varphi_1 \vee... \vee \varphi_n\bigr)$.
Now distribute 
the existential quantifier prefix $\exists y_2... \exists y_k$
over the disjunctions, obtaining the formula
$\exists y_2... \exists y_k\,\varphi_1 \vee... \vee \exists y_2... \exists y_k\, \varphi_n$.
Now consider the formula $\varphi_j$.
Assume first that $\varphi_j$ is of the type $\alpha \wedge \psi$,
where $\alpha$ is a non-empty conjunction
of atoms and negated atoms of the
arity $m\geq 2$,
and $\psi$ is a non-empty conjunction  
of formulae that have at most one free variable.
Let $z_2,...,z_p \in Y$
denote the variables in $Y\setminus\{y_1\}$
that occur in $\alpha$.
Notice that $p = m$ if and only if $y_1$
occurs in $\alpha$.
Let $z_1$ denote $y_1$.
Let $z_{{p+1}},..., z_{{k}}\in Y$ be the variables in $Y\setminus \{z_1,...,z_p\}$.
Notice that the formula $\psi$ is equivalent to the conjunction
$\psi_1(z_1)\wedge ... \wedge \psi_k(z_k) \wedge \beta$, where 
each formula $\psi_i(z_i)$ is the conjunction of exactly all conjuncts of $\psi$
with the free variable $z_i$, in the case such
conjuncts exist, and $\psi_i(z_i)$ is the formula $\top$
otherwise; the formula $\beta$
is the conjunction of the conjuncts of $\psi$ without free variables.
The formula $\varphi_j$ is equivalent to the formula
$
\exists z_2...\exists z_p\bigl(\, \alpha
\wedge \psi_1(z_1) \wedge ... \wedge \psi_p(z_p)\, \bigr)
\wedge \exists z_{p+1}\psi_{p+1}(z_{p+1})\, \wedge ... \wedge\,
\exists z_k\psi_k(z_k)\, \wedge\, \beta.
$
Notice that for each $i$, the formula $\exists z_i\psi_i(z_i)$ is a
$\mathrm{DUF}_1$-formula if $\psi_i(z_i)$ is.
Consider the formula
$\gamma := \exists z_2...\exists z_p\bigl(\, \alpha
\wedge \psi_1(z_1) \wedge ... \wedge \psi_p(z_p)\, \bigr).$
The formula $\alpha$ is either equivalent to $\bot$,
or equivalent to a non-empty disjunction $\delta_1\vee...\vee\delta_l$,
where each $\delta_i$ denotes a conjunction over
some uniform $m$-ary $\tau_{\chi}$-diagram.
Assume first that $\alpha$ is equivalent to $\delta_1\vee...\vee\delta_l$.
Therefore the formula $\gamma$
is equivalent to the disjunction
$\exists z_2... \exists z_p\bigl(\, \delta_1\wedge\psi_1(z_1) \wedge ... \wedge \psi_p(z_p)\, \bigr)
\vee...\vee\exists z_2... \exists z_p\bigl(\, \delta_l \wedge\psi_1(z_1) \wedge ... \wedge \psi_p(z_p)\, \bigr).$
Notice that the disjunct $\exists z_2... \exists z_p\bigl(\, \delta_i\wedge\psi_1(z_1) \wedge ... \wedge \psi_p(z_p)\, \bigr)$
is a $\mathrm{DUF}_1$-formula if the formulae $\psi_1(z_1),...,\psi_p(z_p)$ are;
we may need to use the formation rule (iv) of $\mathrm{DUF}_1$
in addition to rule (iii) if $\exists z_2... \exists z_p\bigl(\, \delta_i\wedge\psi_1(z_1) \wedge ... \wedge \psi_p(z_p)\, \bigr)$
does not contain the free variable $z_1$.
In the case $\alpha$ is equivalent to $\bot$,
then $\gamma$ is equivalent to $\bot$.
%
%but we wish to ensure that the translation from $\mathrm{UF}_1$
%
%to $\mathrm{DUF}_1$ can be done such that the translated
%
%formula has exactly the same free variables and non-logical symbols as the original one.
%
%Thus, the formula $\gamma$ corresponds to, say, 
%
%$\exists z_2...\exists z_p(\delta_i \wedge \top)\, \wedge\, \neg \exists z_2...z_p(\delta_i \wedge \top)$,
%
%where $\delta_i$ is one of the disjuncts of $\delta_1\vee ... \vee \delta_l$.  
%

%
We have now discussed the case where $\varphi_j$ is of the type
$\exists y_2...\exists y_k\bigl(\, \alpha \wedge \psi\, \bigr)$,
where $\alpha$ is a non-empty conjunction of
atoms and negated atoms of some arity higher than one, and
$\psi$ is a non-empty conjunction of formulae with at most
one free variable.
The case where $\varphi_j$ is 
$\exists y_2...\exists y_k\, \alpha$,
%
%where $\alpha$ is a non-empty conjunction of atoms of some
%
%arity higher than one,
%
can be reduced to the case already discussed by
considering the formula $\exists y_2...\exists y_k\bigl(\, \alpha \wedge \top\, \bigr)$.
Assume thus that $\varphi_j$ is the formula
$\exists y_2...\exists y_k\, \psi$,
where $\psi$ is some conjunction
$\psi_1(y_1)\wedge...\wedge\psi_k(y_k) \wedge \beta$,
where the formulae $\psi_i(y_i)$
have at most one free variable, and $\beta$ has no free variables. 
Now $\varphi_j$ is equivalent to the formula
$\psi_1(y_1)\wedge\exists y_2\psi_2(y_2)\wedge...\wedge\exists y_k\psi_k(y_k) \wedge \beta$.
Each conjunct $\exists y_i \psi_1(y_i)$ is a $\mathrm{DUF}_1$-formula
if $\psi_i(y_i)$ is. 
All other cases concerning the translation from $\mathrm{UF}_1$
to $\mathrm{DUF}_1$
are straightforward.
\end{proof}

\section{Proofs for Section \ref{decidabilitysection}}
%
%\textbf{Proofs for Section \ref{decidabilitysection}}
%
\textbf{Proof of Lemma \ref{secondlemmaofwarmup}.}
%
%\begin{proof}
%
We establish the claim of the lemma by showing that
$$\T,\frac{(w,t)}{x} \models\,
\psi_{\mathit{total}}\, \wedge\, \psi_{\mathit{uniq}}
\, \wedge\, \psi_{\mathit{local}}\, \wedge\,
\psi_{\mathit{sub}}\, \wedge\, P_{\psi}(x).$$
To show that $\T \models\,
\psi_{\mathit{total}}$, let $\bigl((u_1,t_1),...,(u_k,t_k)\bigr)\, \in\, (\mathit{Dom}(\T))^k$,
where $k\in \{2,...,\mathcal{M}\}$.
We need to show that $\bigl((u_1,t_1),...,(u_k,t_k)\bigr)$
satisfies the formula $\mathit{Cons}_{\delta}(x_1,...,x_k)$ for
some $\delta\in\Delta_k$.
Consider the tuple $(u_1,...,u_k) \in M^k$.
Let $\eta$ be the unique
standard uniform $k$-ary $V_{\psi}$-diagram $\eta$ such
that $(u_1,...,u_k) \in\ \Vert \eta \Vert^{\M}$.
Let $p\in\{2,...,\mathcal{M}\}$, $p\geq k$.
Let $\rho\in\Delta_p$.
Let $f:\{1,...,p\}\rightarrow\{1,...,k\}$
be a surjection, and assume that $\eta\, \leq_f\, \rho$.
Thus $(u_{f(1)},...,u_{f(p)})\in \Vert \rho \Vert^{\M}$.
In order to conclude that $\T\models\psi_{\mathit{total}}$,
we need to show that
$\T,\frac{\bigl((u_1,t_1),...,(u_k,t_k)\bigr)}{(x_1,...,x_k)}
\models
\mathit{PreCons}_{\rho}(x_{f(1)},...,x_{f(p)})$.
Therefore we assume that
$\T,\frac{\bigl((u_1,t_1),...,(u_k,t_k)\bigr)}{(x_1,...,x_k)}
\models
P_{\chi_1}(x_{f(1)})\wedge...\wedge P_{\chi_p}(x_{f(p)}).$
Thus we have 
$(\M,u_{f(i)})
\models
\chi_{i}$ for each $i\in\{1,...,p\}$.
As $(u_{f(1)},...,u_{f(p)})\in \Vert \rho \Vert^{\M}$,
we therefore have $u_{f(1)}\, \in\, \Vert\langle\rho\rangle(\chi_1,...,\chi_p)\Vert^{\M}$.
Thus $(u_{f(1)},t_{f(1)})
\in P_{\langle \rho \rangle(\chi_1,...,\chi_p)}^{\T}$,
whence 
$\T,\frac{\bigl((u_1,t_1),...,(u_k,t_k)\bigr)}{(x_1,...,x_k)}
\models
P_{\langle \rho \rangle(\chi_1,...,\chi_p)}(x_{f(1)}).$
Therefore $\T\models \psi_{\mathit{total}}$.
It is immediate by the definition of
the domain of $\T$ and the predicates $P_{t}^{\T}$,
where $t$ is a torus point, that $\T\models\psi_{\mathit{uniq}}$.
To show that $\T\models\psi_{\mathit{local}}$,
assume $\T,\frac{(u,t)}{x}\models\mathit{Local}_{\delta}(x)$
for some $k$-ary diagram $\delta\in\Delta$.
Thus $(u,...,u)_k\in\Vert\delta\Vert^{\M}$.
To show that $\T,\frac{(u,t)}{x}\models\mathit{PreCons}_{\delta}(x,...,x)_k$,
let $\langle\delta\rangle(\chi_1,...,\chi_k)\in\mathrm{SUB}_{\psi}$ and
assume that $\T,\frac{(u,t)}{x}\models P_{\chi_1}(x)\wedge...\wedge P_{\chi_k}(x)$.
Therefore $u\in\Vert\chi_i\Vert^{\M}$ for each $i\in\{1,...,k\}$,
whence $u\in\Vert\langle\delta\rangle(\chi_1,...,\chi_k)\Vert^{\M}$.
Thus $(u,t)\in P_{\langle\delta\rangle(\chi_1,...,\chi_k)}^{\T}$, as required.
%

%
\begin{comment}
%
The fact that $\T \models \psi_{\mathit{sub}}$
%
follows by the definition of the predicates $P_{\chi}^{\T}$,
%
\end{comment}
%
The non-trivial part in proving that  $\T \models \psi_{\mathit{sub}}$
involves showing that $\T \models \psi_{\langle \delta \rangle(\chi_1,...,\chi_k)}$
for formulae of the type $\langle\delta\rangle(\chi_1,...,\chi_k)$.
This follows directly by Lemma \ref{firstlemmaofwarmup}, since
$P_{\langle\delta\rangle(\chi_1,...,\chi_k)}^{\T}
\ =\ \ \Vert \langle \delta \rangle (\chi_1,...,\chi_k)
\Vert^{\M}\times\ \mathit{Dom}(\T).$
Since $(\M,w)\models \psi$ and $P_{\psi}^{\T}\, =\ \Vert\psi\Vert^{\M}
\times\ \mathit{Dom}(\T)$, we have
$\T,\frac{(w,t)}{x} \models P_{\psi}(x)$.
\qed\\
\noindent
\textbf{Proof of Lemma \ref{thirdlemmaofwarmup}.}
We establish the claim by induction on the structure of $\alpha$.
For all atomic formulae $S \in \mathrm{SUB}_{\psi}$, the claim follows
directly from the definition of the relations $S^{\B}$ on tuples
that span a singleton set.
The cases where $\alpha$ is
of form $\neg \beta$ or $(\beta \wedge \gamma)$
are straightforward since $\A\models \psi_{\mathit{sub}}$.
%

%
%Assume then that $\A,\frac{u}{x}\models P_{{}_R}(x)$.
%
%We immediately conclude that $(\B,u)\models R$ by
%
%the definition of $R^{\B}$ on tuples $(u,...,u)_n$.
%

%
Define $u_1 := u$ and $x_1 := x$.
Assume that $\B,\frac{u_1}{x_1}\models \langle\delta\rangle(\chi_1,...,\chi_k)$,
where $\langle\delta\rangle(\chi_1,...,\chi_k)\in\mathrm{SUB}_{\psi}$.
Thus $(u_1,...,u_k) \in\ \Vert \delta \Vert^{\B}$
for some tuple 
$(u_1,...,u_k)$ such that $u_i \in\ \Vert \chi_i \Vert^{\B}$ for each $i\in\{1,...,k\}$.
Now, for each $i\in\{1,...,k\}$,
we have $P_{\chi_i}^{\A} =\ \Vert \chi_i \Vert^{\B}$
by the induction hypothesis, and therefore
$u_i \in P_{\chi_i}^{\A}$.
By Lemma \ref{lemmawhatever}, we have
$\A,\frac{(u_1,...,u_k)}{(x_1,...,x_k)}
\models \mathit{PreCons}_{\delta}(x_1,...,x_k).$
By the definition of the formula $\mathit{PreCons}_{\delta}(x_1,...,x_k)$,
we conclude that
%
%$u_1\, \in P_{\langle \delta\rangle (\chi_1,...,\chi_k)}^{\A}$.
%
$\A,\frac{u_1}{x_1}\models P_{\langle \delta\rangle (\chi_1,...,\chi_k)}(x_1)$.
For the converse, assume $\A,\frac{u_1}{x_1}
\models P_{\langle\delta\rangle(\chi_1,...,\chi_k)}(x_1)$.
As $\A\models \psi_{\langle\delta\rangle(\chi_1,...,\chi_k)}$, we have
%
%
%
%\begin{align*}
%
$\A,\frac{u_1}{x_1}\ \models\ \exists x_2...\exists x_k
\bigl( \mathit{Diag}_{\delta}(x_1,...,x_k) \wedge P_{\chi_1}(x_1)
\wedge...\wedge P_{\chi_k}(x_k)\bigr).$
%
%\end{align*}
%
%
%
Hence there exists some
tuple $(u_1,...,u_k)$
such that $u_i \in P_{\chi_i}^{\A}$ for each $i$ and
$\A,\frac{(u_1,...,u_k)}{(x_1,...,x_k)}
\models \mathit{Diag}_{\delta}(x_1,...,x_k).$
By Lemma \ref{lemmasomething}, 
we have $(u_i,...,u_k)\in\ \Vert \delta \Vert^{\B}$. 
As $\Vert \chi_i \Vert^{\B} = P_{\chi_i}^{\A}$ for each $i$ by 
the induction hypothesis, we
conclude that $(\B,u_1)\models \langle \delta \rangle (\chi_1,...,\chi_k)$.
Assume that $(\B,u)\models \langle E\rangle \chi$,
where $\langle E \rangle \chi \in\mathrm{SUB}_{\psi}$.
Thus $(\B,v)\models \chi$
for some $v$, whence $\A,\frac{v}{y}\models P_{\chi}(y)$
by the induction hypothesis.  Thus $\A \models \exists y P_{\chi}(y)$.
As $\A\models\psi_{\mathit{sub}}$,
we have $\A,\frac{u}{x}\models P_{\langle E\rangle \chi}(x)$.
Assume that $\A,\frac{u}{x}\models P_{\langle E\rangle \chi}(x)$.
As $\A\models\psi_{\mathit{sub}}$,
we have $\A\models \exists y P_{\chi}(y)$,
whence $\A,\frac{v}{y}\models P_{\chi}(y)$
for some $v$. By the induction hypothesis,
we have $(\B,v)\models \chi$,
whence $(\B,u)\models \langle E\rangle \chi$.
\qed
\section{Arguments concerning undecidable extensions}\label{undec}
%
%for Section \ref{undecidability}}
%
%
%

%
%\subsection{Tiling problems}
%

Let us recall the tiling problem of the infinite grid $\NN\times\NN$.
A tile is a mapping $t: \{R,L,T,B\}\to C$, where $C$ is a countably infinite set of colours.
We use the subscript notation $t_X:=t(X)$ for $X\in\{R,L,T,B\}$.
Intuitively, $t_R$, $t_L$, $t_T$ and $t_B$ 
are the colours of the right edge, left edge, top edge and 
bottom edge of the tile $t$, respectively.  

Let $\cT$ be a finite set of tiles. A $\cT$-tiling of $\NN\times\NN$ is a function
$f:\NN\times\NN\to\cT$ that satisfies the following horizontal and vertical 
tiling conditions:
\begin{description}
\item[$(T_H)$] For all $i,j\in\NN$, if $f(i,j)=t$ and $f(i+1,j)=t'$, then $t_R=t'_L$.
\item[$(T_V)$] For all $i,j\in\NN$, if $f(i,j)=t$ and $f(i,j+1)=t'$, then $t_T=t'_B$.
\end{description}
Thus, $f$ is a proper tiling if and only if the colors on the matching edges 
of any two neighbouring tiles coincide.
The tiling problem for the grid $\NN\times\NN$ asks
whether for a finite set $\cT$ of tiles, there exist a $\cT$-tiling of
$\NN\times\NN$.
It is well known that this problem is undecidable ($\Pi^0_1$-complete).
Using the tiling problem, it is straightforward to prove the 
following proposition.

%
%\subsection{General one-dimensional fragment is  $\Pi^0_1$-complete}
%
%
%
\begin{proposition}\label{general}
The satisfiability problem of\, $\mathrm{GF}_1$ is\, $\Pi^0_1$-complete.
%
%and therefore undecidable. 
%
\end{proposition}
\begin{proof}
Let $\tau=\{H,V\}$ be a vocabulary, where $H$ and $V$ are binary relation symbols.
The infinite grid $\NN\times\NN$ can be represented by a $\tau$-structure 
$\G:=(\NN\times\NN,H^\G,V^\G)$, where 
$H^\G:=\{((i,j),(i+1,j))\, |\,  i,j\in\NN\}$ and
$V^\G:=\{((i,j),(i,j+1))\, |\,  i,j\in\NN\}$.
Let $\Gamma$ be the conjunction of the three $\tau$-sentences
$\eta_H:= \forall x\exists y\,H(x,y)$,
$\eta_V:= \forall x\exists y\,V(x,y)$, and
$\eta_{\mathit{Com}}:= 
\forall x\forall y\forall z\forall w\,\bigl((H(x,y)\land V(x,z)\land H(z,w))\to V(y,w)\bigr)$.
It is easy to see that $\eta_H$, $\eta_V$ and $\eta_{\mathit{Com}}$ are  
in $\mathrm{GF}_1$. 
% 
% Strictly speaking, only equivalent with sentences in $\mathrm{GF}_1$.

%
%We prove first that $\G$ can be homomorphically
%
%embedded in any model of $\Gamma$. 
%

%\begin{lemma}\label{homom}
%
It is straightforward to show that if $\M$ is a $\tau$-model such that $\M\models\Gamma$,
then there exists a homomorphism $h: \G\to \M$.
%
%\end{lemma} 
%
%\begin{proof}
%
%Straightforward. 
%                                                      
%\end{proof}
%

%
Let $\cT$ be a set of tiles. We simulate tiles by
unary relation symbols $P_t$ for each $t\in\cT$. 
We denote the corresponding vocabulary $\tau\cup\{P_t\, | \,  t\in\cT\}$
by $\sigma_\cT$. The tiling conditions $(T_H)$ and $(T_V)$ can 
be expressed by the $\sigma_\cT$-sentences
$
	\psi_H:=\forall x\forall y\bigwedge_{t,t'\in\cT,\;  t_R\not=t'_L}
	(P_t (x)\land P_{t'}(y))\to \lnot H(x,y)$ and
$\psi_V:=\forall x\forall y\bigwedge_{t,t'\in\cT,\;  t_T\not=t'_B}
	(P_t (x)\land P_{t'}(y))\to \lnot V(x,y)$.
Let  $\Psi_\cT:=\psi_H\land\psi_V\land\psi_{\mathit{part}}$, 
where $\psi_{\mathit{part}}$ is a sentence saying that every element
is in exactly one of the relations $P_t$, $t\in\cT$.
Clearly $\psi_{\mathit{part}}$
can be expressed in $\mathrm{GF}_1$.
It is straightforward to show
that the sentence $\Gamma\land\Psi_\cT$ is satisfiable if and only if
$\NN\times\NN$ is $\cT$-tilable.
%
\begin{comment}
%
Assume first that $\M$ is a $\tau\cup\sigma_\cT$-model such that 
%
$\M\models\Gamma\land\Psi_\cT$. By Lemma \ref{homom}, there is a 
%
homomorphism $h:\G\to\M$.
%
Since $\M\models\psi_{\mathit{part}}$, for each 
%
$a\in M$ there is exactly one $t\in\cT$ such that $a\in P_t^\M$.
%
Thus, we can define a function $f:\NN\times\NN\to\cT$ by setting
%
$f(i,j)=t$, where $t$ is the unique tile such that $h(i,j)\in P_t^\M$.
% 
We will now show that $f$ is a tiling of $\NN\times\NN$. 
%

To  see this, assume that $f(i,j)=t$ and $f(i+1,j)=t'$ 
%
for some $i,j\in\NN$. Then 
%
$(h(i,j),h(i+1,j))\in H^\M$, $h(i,j)\in P_t^\M$ and $h(i+1,j)\in P_{t'}^\M$.
%
Since $\M\models\psi_H$, it follows that $t_R=t'_L$.
%
In the same way we see that if $f(i,j)=t$ and $f(i,j+1)=t'$,
%
then $t_T=t'_B$.
%

To prove the converse implication, assume that $f:\NN\times\NN\to\cT$ is a 
%
$\cT$-tiling of the grid $\G$. Define for each $t\in\cT$ the relation
%
$P_t^\G:=\{u\in\NN\times\NN\, |\, f(u)=t\}$.
%
It is now easy to see that 
%
$\M\models\Gamma\land\Psi_\cT$, where $\M$ is the model 
%
$(\NN\times\NN,H^\G,V^\G,(P_t^\G)_{t\in\cT})$.
%
\end{comment}
%
%
%It follows from Lemma \ref{3-tiling} that the satisfiability problem of
%
%the general one-dimensional fragment $\mathcal{GF}_1$ is undecidable.
%
%
%
Since the sentence $\Gamma\land\Psi_\cT$ is in $\mathrm{GF}_1$
for each finite set $\cT$ of tiles,
the tiling problem is effectively reducible to 
the satisfiability problem of $\mathrm{GF}_1$. Hence the satisfiability problem is
$\Pi^0_1$-hard. On the other hand, $\mathrm{GF}_1$ is a fragment of
first-order logic, whence its satisfiability problem is in~$\Pi^0_1$.
\end{proof}
%
%
%
%

%
%\begin{theorem}\cite{IEEEonedimensional:berger, IEEEonedimensional:harel}
%
%The tiling problem for the grid $\NN\times\NN$ is 
%
%$\Pi^0_1$-complete.
%
%\end{theorem}
%

%
%___________________2-dim-tiling____________
%
%\subsection{Uniform $2$-dimensional fragment is undecidable}
%

%
Let $\tau_+=\{H_+,V_+,S\}$ be a vocabulary, where $H_+$ and $V_+$ 
are ternary relation symbols and $S$ is a binary relation symbol.
We will represent the infinite grid $\NN\times\NN$ as
a $\tau_+$-structure $\G_+:=(\NN,H_+^{\G_+},V_+^{\G_+},S^{\G_+})$, where
$H_+^{\G_+}:=\{(i,i+1,j)\, |\,  i,j\in\NN\}$,
$V_+^{\G_+}:=\{(i,j,j+1)\, |\,  i,j\in\NN\}$, and
$S_+^{\G_+}:=\{(i,i+1)\, |\,  i\in\NN\}$.
Notice that $(u,v,w)\in V_+^{\G_+}$
iff $(u,v)$ connects to $(u,w)$ via the
vertical successor $V^{\G}$
of the standard Cartesian grid $\G$
defined in the proof of Proposition \ref{general}.
On the other hand, $(u,v,w)\in H_+^{\G_+}$
iff $\bigl((u,w),(v,w)\bigr)\in H^{\G}$.
We shall next form a $\tau_+$-sentence $\Gamma_+$ of
$\mathrm{SUF}_2$ such that $\G_+\models \Gamma_+$,
and there is a homomorphism from $\G_+$ to any model of $\Gamma_+$.
%
%For this purpose, we define the following three $\mathrm{SUF}_2$-sentences:
%
Define $\Gamma_+$ to be the conjunction of the formulae
$\theta_S:=\;\forall x\exists y\, S(x,y),$
$\theta_H:=\;\forall x_1\forall x_2\,(S(x_1,x_2)\to\forall y\,H_+(x_1,x_2,y)),$ and
$\theta_V:=\;\forall y_1\forall y_2\,(S(y_1,y_2)\to\forall x\,V_+(x,y_1,y_2)).$
%
%Let $\Gamma_+$ be the conjunction of these three sentences.
%
%It is now easy to establish the following lemma.
%
%
%
%
\begin{lemma}\label{appendixlemma}
If $\M$ is a $\tau_+$-model such that $\M\models\Gamma_+$,
then there exists a homomorphism $h: \G_+\to \M$. 
\end{lemma}
\begin{proof}
We define a function $h:\NN\to M$ by recursion as follows.
Choose an arbitrary point $a_0\in M$, and set 
$h(0):=a_0$.
Assume that  $h(i)=a$ has been defined.
Since $\M\models\theta_S$, there is
$b\in M$ such that $(a,b)\in S^\M$. Define $h(i+1):=b$.
Observe first that $(h(i),h(i+1))\in S^\M$
for each $i\in\NN$. Furthermore, since $\M\models\theta_H\land\theta_V$,
we have
$
	(h(i),h(i+1),h(j))\in H_+^\M \;\text{ and }\; (h(i),h(j),h(j+1))\in V_+^\M
$
for all $i,j\in\NN$. Thus $h$ is a homomorphism $\G_+\to\A$.
\end{proof}
%
%
%
%\begin{proof}
%
%See the appendix.
%
%\end{proof}
%
%We can now prove the following theorem
%
%
%
\begin{theorem}
The satisfiability problem of $\mathrm{SUF}_2$ is
$\Pi_1^0$-complete.
\end{theorem}
\begin{proof}
By Lemma \ref{appendixlemma}, we know that
if 
$\M$ is a $\tau_+$-model such that $\M\models\Gamma_+$,
then there exists a homomorphism $h: \G_+\to \M$.
(We also have $\G_+\models\Gamma_+$.)

Let $\cT$ be a set of tiles.
This time we simulate tiles by fresh ternary relation symbols
$P_{X,t}$, where $X\in\{R,L,T,B\}$ and $t\in \cT$. Let 
$\rho_\cT:=\tau_+\cup\{P_{X,t}\, |\,  X\in\{R,L,T,B\}, t\in \cT\}$ be
the corresponding vocabulary.
The idea here is that if $(a,b,c)\in P_{R,t}$ and $(a,b,c)\in P_{L,t'}$,
then the right edge of $(a,c)$ is coloured with $t_R$ and  
the left edge of $(b,c)$ is coloured with $t'_L$;
recall that $(a,b,c)\in H_+^{\G_+}$ means that $\bigl((a,c),(b,c)\bigr)\in H^{\G}$.
Similarly, if $(a,b,c)\in P_{T,t}$ and $(a,b,c)\in P_{B,t'}$ , then
the top edge of $(a,b)$ is coloured with $t_T$  and 
the bottom edge of $(a,c)$ is coloured with $t'_B$.
Thus, we can express the tiling conditions $(T_H)$ and $(T_V)$
by the following $\mathrm{SUF}_2$-sentences:\smallskip \\
%
%\begin{multline*}
$
	\varphi_H:=\forall x_1\forall x_2\forall y\bigwedge\limits_{t,t'\in\cT,\;  t_R\not=t'_L}\\
	\text{ }\ \ \ \ \Bigl(\bigl(P_{R,t}(x_1,x_2,y)\land P_{L,t'}(x_1,x_2,y)\bigr)
           \to \lnot H_+(x_1,x_2,y)\Bigr),\\
$
%\end{multline*}
%
$
	\varphi_V:= \forall x\forall y_1\forall y_2\bigwedge\limits_{t,t'\in\cT,\;  t_T\not=t'_B}\\
	\text{ }\ \ \ \ \Bigl(\bigl(P_{T,t}(x,y_1,y_2)\land P_{B,t'}(x,y_1,y_2)\bigr)
              \to \lnot V_+(x,y_1,y_2)\Bigr).
$

We also need a sentence $\varphi_{\mathit{prop}}$ stating that every pair 
$(a,b)$ is tiled by exactly one $t\in\cT$. This amounts to stating, firstly, that
the interpretation of each symbol $P_{R,t}$
depends only on the first and the last variable:
$
	\bigwedge_{t\in\cT}\forall x_1\forall y\, 
	(\exists x_2 \,P_{R,t}(x_1,x_2,y) \to\forall x_2 \,P_{R,t}(x_1,x_2,y)),
$
and analogously for $P_{L,t}$, $P_{T,t}$ and $P_{B,t}$. Secondly, the four
colors of each pair  correspond to the same tile, meaning that the formula
$
	\bigwedge_{t\in\cT}\forall x_1\forall y\, 
	(\exists x_2 \,P_{R,t}(x_1,x_2,y)\leftrightarrow\exists x_2 \,P_{L,t}(x_2,x_1,y))
$
holds, and similar conditions for the other pairs 
$(P_{X,t},P_{Y,t})$ hold.
Thirdly, for each $X\in\{L,R,B,T\}$,
every triple is in exactly one of the relations $P_{X,t}$, $t\in\cT$.
Clearly there is such a sentence $\varphi_{\mathit{prop}}$ in $\mathrm{SUF}_2$.
Let $\Phi_\cT$ be the conjunction of the sentences $\varphi_H$, 
$\varphi_V$ and $\varphi_{\mathit{prop}}$.
Thus we have established
that the sentence $\Gamma_+\land\, \Phi_\cT$ is satisfiable if and only if
$\NN\times\NN$ is $\cT$-tilable.
Hence we conclude that $\mathrm{SUF}_2$ is $\Pi_1^0$-complete.
\end{proof}
%
%\begin{proof}
%
%Since the sentence $\Gamma_+\land\Phi_\cT$ is in $\mathrm{SUF}_2$
%
%for each set $\cT$ of tiles,
%
%by Lemma~\ref{+tiling} the tiling problem is effectively reducible to 
%
%the satisfiability problem of $\mathrm{SUF}_2$. Hence the latter is 
%
%$\Pi^0_1$-hard. On the other hand, $\mathrm{SUF}_2$ is a fragment of
%
%first-order logic, whence its satisfiability problem is in~$\Pi^0_1$.\qed
%

%
\section{Expressivity}
\begin{theorem}
$\mathrm{UF}_1$ is incomparable in expressivity with both
two-variable logic with counting $(\mathrm{FOC}^2)$ and
guarded negation fragment $(\mathrm{GNFO})$.
\end{theorem}
\begin{proof}
%
%\subsection{Comparing to other decidable fragments}
%
\begin{comment}
%
The \emph{two-variable logic} $\mathrm{FO}^2$ is one of the most studied 
%
decidable fragments of first-order logic. It follows immediately from the 
% 
definition of the uniform one-dimensional fragment that $\mathrm{FO}^2$
%
is contained in $\mathrm{UF}_1$. In fact, if we consider only structures
%
in a binary vocabulary, it is easy to see that $\mathrm{UF}_1$ collapses
%
to $\mathrm{FO}^2$. On the other hand,
%
\end{comment}
%
%
%
%In fact, if we consider only structures
%
%in a binary vocabulary, it is easy to see that $\mathrm{UF}_1$ collapses
%
%to $\mathrm{FO}^2$.
%
The expressivity of 
$\mathrm{FOC}^2$  is seriously limited
when it comes to properties of relations
of arities greater than two. It is straightforward
to show that for example the $\mathrm{UF}_1$-sentence
$\exists x\exists y\exists z\, R(x,y,z)$ is not expressible in $\mathrm{FOC}^2$.
%
%The two-variable logics
%
%$\mathrm{FO}^2$ and $\mathrm{FOC}^2$
%
%do no cope well with relations of arities greater than two.
%
%In fact, it is straightforward to show that the sentence
%
%$\exists x\exists y\exists z\, R(x,y,z)$ is not expressible in $\mathrm{FOC}^2$.
%
%is not expressible even in two-variable 
%
%logic with counting $\mathrm{FOC}^2$.
%
%(It is straightforward to give a formal
%
%proof of this by applying a simple pebble-game argument; see
%
%\cite{IEEEonedimensional:libkin} for the definition of pebble games.)
%
Thus $\mathrm{UF}_1$ is not contained in $\mathrm{FOC}^2$.
It is straightforward to show by using the bisimulation for 
$\mathrm{GNFO}$, provided in \cite{IEEEonedimensional:barany}, that
the $\mathrm{UF}_1$-sentence $\exists x\exists y\,\lnot R(x,y)$
is not expressible in $\mathrm{GNFO}$.
This follows from the fact that structures
$$\bigl(\{a\},\{(a,a)\}\bigr)\text{ and }\bigl(\{a,b\},\{(a,a),(b,b)\}\bigr)$$
are bisimilar in the sense of $\mathrm{GNFO}$.
Thus $\mathrm{UF}_1$ is not contained in $\mathrm{GNFO}$.
The $\mathrm{FO}^2$-sentence $\forall x\forall y(x=y)$ 
cannot be expressed in $\mathrm{UF}_1$. This can be seen
(for example) by observing that the two directions of our
decidability proof together entail that 
satisfiable sentences of the equality-free logic $\mathrm{UF}_1$ can
always be satisfied in a larger model. Thus $\mathrm{UF}_1$ does not
contain $\mathrm{FO}^2$.
%

%
% Claim: $\mathrm{UF}_1$ is a canonical extension of $\mathrm{FO}^2$
% to vocabularies of arity $\ge 3$.
%

%
It follows immediately from the 
definition of $\mathrm{UF}_1$ that 
the equality-free fragment of $\mathrm{FO}^2$
is contained in $\mathrm{UF}_1$.
In fact, it is easy to prove that in restriction models with
relation symbols of arities at most two, the expressivities
of $\mathrm{UF}_1$ and the identity-free 
fragment of $\mathrm{FO}^2$ coincide. (Consider for
example the translation from $\mathrm{UF}_1$ to
$\mathrm{MUF}_1$ in the case of such vocabularies.)
To see that $\mathrm{UF}_1$ does not contain
$\mathrm{GNFO}$, consider for example the $\mathrm{GNFO}$-sentence
$\exists x\exists y \exists z (Rxy\wedge Ryz \wedge Rzx)$. It is straightforward
to show (by a pebble game argument, see \cite{IEEEonedimensional:libkin}), that this
property is not expressible in $\mathrm{FO}^2$. Since
$\mathrm{UF}_1$ is contained in $\mathrm{FO}^2$ when attention
is restricted to models with only binary relations, we conclude that
$\mathrm{UF}_1$ does not contain $\mathrm{GNFO}$.
\end{proof}

\end{document}